\documentclass[11pt,english]{article}

\usepackage[T1]{fontenc}
\usepackage[latin9]{inputenc}
\usepackage{geometry}
\usepackage{babel}
\usepackage{amsmath}
\usepackage{amsthm}
\usepackage{amssymb}
\usepackage[unicode=true,pdfusetitle,
 bookmarks=true,bookmarksnumbered=false,bookmarksopen=false,
 breaklinks=false,pdfborder={0 0 0},pdfborderstyle={},backref=false,colorlinks=false]
 {hyperref}

\makeatletter
\date{}

\usepackage{appendix}

\usepackage{bm}

\usepackage[nameinlink,capitalise,noabbrev]{cleveref}
\crefname{appsec}{Appendix}{Appendices}

\theoremstyle{plain}
\newtheorem{thm}{\protect\theoremname}
\crefname{thm}{Theorem}{Theorems}
\theoremstyle{definition}

\theoremstyle{plain}
\newtheorem{prop}[thm]{\protect\propositionname}

\newtheorem{lem}[thm]{\protect\lemmaname}
\crefname{lem}{Lemma}{Lemmas}

\theoremstyle{definition}

\theoremstyle{plain}

\theoremstyle{plain}
\crefformat{equation}{#2(#1)#3}

\let\originalleft\left
\let\originalright\right
\renewcommand{\left}{\mathopen{}\mathclose\bgroup\originalleft}
\renewcommand{\right}{\aftergroup\egroup\originalright}
\usepackage{pgfplots}
\usetikzlibrary{pgfplots.groupplots}
\usepackage{verbatim}

\makeatletter
\renewcommand*{\UrlTildeSpecial}{%
  \do\~{%
    \mbox{%
      \fontfamily{ptm}\selectfont
      \textasciitilde
    }%
  }%
}%
\let\Url@force@Tilde\UrlTildeSpecial
\makeatother

\usepackage{tikz}
\tikzstyle{vertex}=[circle,draw=black,fill=black,inner sep=0,minimum size=0.2cm,text=white,font=\footnotesize]
\tikzset{every loop/.style={min distance=50,in=50,out=130,looseness=7}}

\usepackage[labelfont=bf,labelsep=period]{caption}

\global\long\def\mk#1{\textcolor{blue}{\textbf{[MK comments:} #1\textbf{]}}}

\def\mk#1{}

\makeatother

\providecommand{\claimname}{Claim}
\providecommand{\conjecturename}{Conjecture}
\providecommand{\corollaryname}{Corollary}
\providecommand{\definitionname}{Definition}
\providecommand{\examplename}{Example}
\providecommand{\factname}{Fact}
\providecommand{\lemmaname}{Lemma}
\providecommand{\propositionname}{Proposition}
\providecommand{\remarkname}{Remark}

\providecommand{\theoremname}{Theorem}

\begin{document}
\title{
The random $k$-matching-free process}

\author{Michael Krivelevich\thanks{School of Mathematical Sciences, Raymond and Beverly Sackler Faculty
of Exact Sciences, Tel Aviv University, Tel Aviv 6997801, Israel. Email: \href{krivelev@post.tau.ac.il} {\nolinkurl{krivelev@post.tau.ac.il}}. Research supported in part by USA-Israel BSF grant 2014361 and by ISF grant 1261/17.}\and Matthew Kwan\thanks{Department of Mathematics, ETH, 8092 Z\"urich, Switzerland. Email: \href{mailto:matthew.kwan@math.ethz.ch} {\nolinkurl{matthew.kwan@math.ethz.ch}}.}\and
Po-Shen Loh\thanks{Department of Mathematical Sciences, Carnegie Mellon University, Pittsburgh,
PA 15213. Email: \href{mailto:ploh@cmu.edu} {\nolinkurl{ploh@cmu.edu}}.
Research supported by NSF Grant DMS-1201380 and by NSF CAREER Grant DMS-1455125.}\and Benny Sudakov\thanks{Department of Mathematics, ETH, 8092 Z\"urich, Switzerland. Email:
\href{mailto:benjamin.sudakov@math.ethz.ch} {\nolinkurl{benjamin.sudakov@math.ethz.ch}}. Research supported in part by SNSF grant 200021-175573.}}

\maketitle
\global\long\def\QQ{\mathbb{Q}}
\global\long\def\E{\mathbb{E}}
\global\long\def\Var{\operatorname{Var}}
\global\long\def\CC{\mathbb{C}}
\global\long\def\NN{\mathbb{N}}
\global\long\def\ZZ{\mathbb{Z}}
\global\long\def\GG{\mathbb{G}}
\global\long\def\supp{\operatorname{supp}}
\global\long\def\one{\boldsymbol{1}}
\global\long\def\range#1{\left[#1\right]}
\global\long\def\d{\operatorname{d}\!}
\global\long\def\falling#1#2{\left(#1\right)_{#2}}
\global\long\def\im{\operatorname{im}}
\global\long\def\sp{\operatorname{span}}
\global\long\def\rank{\operatorname{rank}}
\global\long\def\sign{\operatorname{sign}}
\global\long\def\mod{\operatorname{mod}}
\global\long\def\id{\operatorname{id}}
\global\long\def\tr{\operatorname{tr}}
\global\long\def\adj{\operatorname{adj}}
\global\long\def\Unif{\operatorname{Unif}}
\global\long\def\Po{\operatorname{Po}}
\global\long\def\Bin{\operatorname{Bin}}
\global\long\def\Ber{\operatorname{Ber}}
\global\long\def\Geom{\operatorname{Geom}}
\global\long\def\Exp{\operatorname{Exp}}
\global\long\def\Hom{\operatorname{Hom}}
\global\long\def\vol{\operatorname{vol}}
\global\long\def\floor#1{\left\lfloor #1\right\rfloor }
\global\long\def\ceil#1{\left\lceil #1\right\rceil }
\global\long\def\cond{\,\middle|\,}
\global\long\def\EGc{G_{\mathrm{clique}}}
\global\long\def\EGs{G_{\mathrm{star}}}
\global\long\def\all{\mathrm{all}}
\global\long\def\fix{\mathrm{fix}}
\global\long\def\sub{\mathrm{sub}}

\begin{abstract}
Let $\mathcal{P}$ be a graph property which is preserved by removal
of edges, and consider the random graph process that starts with the
empty $n$-vertex graph and then adds edges one-by-one, each chosen
uniformly at random subject to the constraint that $\mathcal{P}$
is not violated. These types of random processes have been the subject
of extensive research over the last 20 years, having striking applications
in extremal combinatorics, and leading to the discovery of important probabilistic
tools. In this paper we consider the \emph{$k$-matching-free process},
where $\mathcal{P}$ is the property of not containing a matching
of size $k$. We are able to analyse the behaviour of this process
for a wide range of values of $k$; in particular we prove that if
$k=o\left(n\right)$ or if $n-2k=o\left(\sqrt{n}/\log n\right)$ then
this process is likely to terminate in a $k$-matching-free graph
with the maximum possible number of edges, as characterised by Erd\H os
and Gallai. We also show that these bounds on $k$ are essentially
best-possible, and we make a first step towards understanding the
behaviour of the process in the intermediate regime. 
\end{abstract}

\section{Introduction}

Following Erd\H os and R\'enyi's seminal papers
on random graphs \cite{ER59,ER60}, there has been great interest in many different
kinds of random graphs and random graph processes, with broad applications
to various combinatorial problems and to real-world networks.
The most basic random graph process, introduced by Erd\H os and R\'enyi,
starts with the empty $n$-vertex graph and adds edges one-by-one,
each selected uniformly at random among the edges not used so far.
A particularly important variation of this basic process is the \emph{random
greedy} process. Here a decreasing\footnote{We say a graph property is \emph{decreasing} if it is preserved by
removal of edges, and we say a property is \emph{increasing }if it
is preserved by addition of edges.} property $\mathcal{P}$ is specified, and then edges are added to
the empty $n$-vertex graph one-by-one, chosen uniformly at random
among edges whose addition to the current graph would not violate
$\mathcal{P}$. A specific example of this type of process was first studied by Ruci\'nski
and Wormald \cite{RW} in 1992, and the idea was first discussed in full generality
by Erd\H os, Suen and Winkler \cite{ESW} in 1995.

Since then, a wide range of different types of random greedy processes
have been studied. Perhaps the most famous specific example is the
\emph{triangle-free} \emph{process}, where $\mathcal{P}$ is the property
that a graph does not contain a triangle (see for example \cite{ESW,Boh,FGM}).
More generally, much of the work on random greedy processes has focused
on cases of the \emph{$H$-free} \emph{process}, where\emph{ }$\mathcal{P}$
is the property that a graph does not contain a copy of a specified
graph $H$ (see for example \cite{BR,OT,Wol,BK,Picl,War}). The theory
of $H$-free processes has also been extended to hypergraphs (see
for example \cite{GRW,BFMRS2,BFL,KOT,BMP}). We remark that in all
the aforementioned results $H$ is a fixed ``small'' (hyper)graph,
whose size does not depend on $n$, and therefore the property of
being $H$-free is in some sense a ``local'' constraint. Much less
is known about random greedy processes for more ``global'' properties
$\mathcal{P}$; two notable exceptions are the random greedy\emph{
planar graph process} \cite{GSST}, and the random greedy \emph{$k$-colourable
process} \cite{ESW,KSV}.

There are a variety of different questions one can ask about random
greedy processes. Commonly, one asks about the size and structure
of the final (or almost-final) outcome of such a process. The process
may a.a.s.\footnote{By ``asymptotically almost surely'', or ``a.a.s.'', we mean that
the probability of an event is $1-o\left(1\right)$. Here and for
the rest of the paper, asymptotics are as $n\to\infty$.}\ ``saturate'' and result in a graph with (almost) the maximum possible
number of edges permitted by $\mathcal{P}$, or it may a.a.s.\ result
in a graph with special properties that are useful for applications.
Examples of the former situation include the \emph{bounded-degree
process} \cite{RW} that pioneered the study of random greedy processes,
and the \emph{triangle removal process }\cite{BFL}, which has become
an important tool in the study of Steiner triple systems \cite{Kee,Kwa}.
A celebrated example of the latter situation is the triangle-free
process, which a.a.s.\ produces triangle-free graphs with no large
independent set; analysis of this process led to important breakthroughs
in Ramsey theory \cite{Boh,FGM,BK2}. There are also situations where
the \emph{intermediate} states of a random greedy process are of particular
interest; for example, the intermediate stages of the random \emph{satisfiable
process} \cite{KSV} are a good source of satisfiable formulas with
certain unique properties.

In this paper we take a first look at the behaviour of the $H$-free
process for an important choice of $H$ with non-fixed size. A \emph{$k$-matching}
is a union of $k$ disjoint edges. For any $k$ (which may depend
on $n$), the \emph{$k$-matching-free} \emph{process} is formally
defined as follows. Let $N={n \choose 2}$, and let $e\left(1\right),\dots,e\left(N\right)$
be a uniformly random ordering of the unordered pairs in ${\range n \choose 2}$
(that is, a random ordering of the edges of the complete graph $K_{n}$). This is the distribution obtained by iteratively selecting each $e(t)$ uniformly at random from the previously unseen edges.
Let $G\left(0\right)$ be the empty $n$-vertex graph, and for $1\le t\le N$
define 
\[
G\left(t\right)=\begin{cases}
G\left(t-1\right) & \text{if }G\left(t-1\right)+e\left(t\right)\text{ contains a }k\text{-matching};\\
G\left(t-1\right)+e\left(t\right) & \text{otherwise.}
\end{cases}
\]
In the former case we say $e\left(t\right)$ is \emph{rejected} and
in the latter case we say it is \emph{accepted}. The outcome of this
random process is a $k$-matching-free graph $G\left(N\right)$ which is $k$-matching-\emph{saturated}, meaning that the addition of any edge would create a $k$-matching. We remark that the general notion of saturation in graphs and hypergraphs is of broad interest; see for example the surveys of Bollob\'as \cite[Section~3]{BolChapter} and Faudree, Faudree and Schmitt \cite{FFS}.

The general problem of determining whether a graph is $k$-matching-free
(or, basically equivalently, the problem of determining the size of the largest
matching in a graph) is of broad importance in various different areas
of mathematics, computer science and even computational chemistry.
One of the most basic results in this area is due to Erd\H os and
Gallai \cite{EG}, who proved that the maximum possible number of
edges in a $k$-matching-free $n$-vertex graph is
\[
\max\left\{ {2k-1 \choose 2},\;{k-1 \choose 2}+\left(k-1\right)\left(n-k+1\right)\right\} .
\]
This result falls under the umbrella of \emph{extremal graph theory},
one of the central branches of modern combinatorics (see for example
the book of Bollob\'as \cite{BolEGT}). Up to isomorphism, the extremal
graphs that attain the Erd\H os-Gallai bound are as follows.
\begin{itemize}
\item $\EGc$ is a clique on $2k-1$ vertices with the remaining $n-2k+1$
vertices isolated.
\item $\EGs$ is a clique on $k-1$ vertices, in addition to every possible
edge between this clique and the remaining $n-k+1$ vertices. Equivalently,
$\EGs$ is a star $K_{1,n-k+1}$ with its center vertex ``blown up''
to a $\left(k-1\right)$-clique.
\end{itemize}
As our main result, we find that if $k$ is sufficiently small or
sufficiently large (i.e. sufficiently close to $n/2$), then the $k$-matching-free process a.a.s.\ produces
an Erd\H os-Gallai extremal graph, as follows.
\begin{thm}
\label{thm:o(n)}If $k=o\left(n\right)$ then a.a.s.\ $G\left(N\right)\cong\EGs$.
This is tight; if $k=\Omega\left(n\right)$ 
then $G\left(N\right)\ncong\EGs$ with probability $\Omega\left(1\right)$.
\end{thm}
\begin{thm}
\label{thm:clique}If $k=n/2-o\left(\sqrt{n}/\log n\right)$ then
a.a.s.\ $G\left(N\right)\cong\EGc$. This is essentially tight; if
$k=n/2-\omega\left(\sqrt{n}/\log n\right)$
then a.a.s.\ $G\left(N\right)\ncong\EGc$.
\end{thm}
The proofs of \cref{thm:o(n)} and \cref{thm:clique} are quite different
to each other, and involve quite different methods to those typically
used for studying $H$-free processes. In particular, we do not require
the so-called \emph{differential equation method}. The positive and
negative parts of \cref{thm:o(n)} will be proved separately in \cref{sec:o(n)}
and \cref{sec:negative-o(n)}, and \cref{thm:clique} will be proved
in \cref{sec:clique}. We remark that while we made no particular attempt to consider the case $k=n/2-\Theta\left(\sqrt{n}/\log n\right)$, we expect that our proof of \cref{thm:clique} can be modified to show that in this case $G\left(N\right)\ncong\EGc$ with probability $\Omega(1)$.

The regime where $k=\Omega\left(n\right)$ and $n-2k=\omega\left(\sqrt{n}/\log n\right)$
is significantly more challenging to study. As a first step, we show
that if $k\le\varepsilon n$ for small $\varepsilon$,
then a.a.s.\ $G\left(N\right)$ resembles $\EGs$. Observe that $\EGs$ has independence number $n-k+1$ and $k-1$ vertices of degree $n-1$.
\begin{prop}
\label{thm:eps}For all $k$, a.a.s.\ $G\left(N\right)$ has an independent
set of size $n-(1+O(k/n))k$, and at least $\left(1-O(k/n)\right)k$ vertices with degree $n-1$. 
\end{prop}
The proof of \cref{thm:o(n)} suggests that \cref{thm:eps} is actually far from best possible; we suspect that if $k\le \varepsilon n$ for small $\varepsilon$, the error term $O(k/n)$ can be substantially improved. However, we observe that there is in fact a range of $k$ in which $G\left(N\right)$
does not resemble any extremal $k$-matching-free graph. Observe that $\EGc$ has $n-2k+1$ isolated vertices, and as before $\EGs$ has independence number $n-k+1$.
\begin{prop}
\label{thm:not-extremal}The following hold.
\begin{enumerate}
\item[(1)]{There is a constant $c<1/2$ such that if $k\ge c n$ then a.a.s.\ $G\left(N\right)$ has independence number $n-k-\Omega\left(n\right)$.}
\item[(2)]{If  $k=\Omega(n)$ and $n-2k=\Omega(n)$ then a.a.s.\ $G\left(N\right)$ has $n-2k-\Omega(n)$ isolated
vertices.}
\end{enumerate}
\end{prop}
That is to say, there is a range of $\Theta(n)$ values of $k$ for which the outcome of the $k$-matching-free process is typically substantially different from the Erd\H os-Gallai extremal graphs, in the sense that edges incident to an $\Omega(1)$-proportion of its vertices must be changed to arrive at either $\EGs$ or $\EGc$. We will give simple proofs of \cref{thm:eps} and \cref{thm:not-extremal} in \cref{sec:eps} and \cref{sec:not-extremal}, respectively. Also, we remark that for \cref{thm:not-extremal}, we can take $c=1/2-e^{-13}/2$, but no effort was made to optimise this constant.

Finally, recall that a vertex cover of a graph is a set of vertices such that every edge in the graph is incident to one of the vertices of this set. The problem of finding a maximum matching in a graph is in a certain sense dual to the problem of finding a minimum vertex cover, and the matching number (maximum size of a matching) and vertex cover number (minimum size of a vertex cover) are often considered together. Therefore one might naturally consider the restricted covering process $G^\mathrm{vc}(1),\dots,G^\mathrm{vc}(N)$, where we accept an edge $e(t)$ if and only if the vertex cover number would stay below $k$. However, in sharp contrast to the $k$-matching-free process, this restricted covering process exhibits quite trivial behaviour. One can easily check that, up to isomorphism, $\EGs$ is the only graph which is saturated with respect to the property of having vertex cover number less than $k$, so we will always have $G^\mathrm{vc}(N)\cong\EGs$.

\subsection{Notation}

For a probability distribution $\mathcal L$, we write $X\in \mathcal L$ to denote that a random element has distribution $L$. We write $\GG(n,m)$ for the distribution of a uniformly random $m$-edge subset of $K_n$ (this is known as the Erd\H{o}s-R\'enyi random graph), and we use the same notation $\GG\left(n,p\right)$ for the \emph{binomial} random graph where each edge of $K_n$ is present independently with probability $p$. Also, for $0\le t\le N=\binom n2$, let $G^{\all}\left(t\right)$ be the graph with all the edges $e\left(1\right),\dots,e\left(t\right)$. This graph has precisely the Erd\H{o}s-R\'enyi distribution $\GG(n,t)$.

For a real number $x$, the floor and ceiling functions are denoted $\floor{x}=\max\{i\in\mathbb{Z}:i\le x\}$ and $\ceil{x}=\min\{i\in\mathbb{Z}:i\ge x\}$. For a positive integer $i$, we write $[i]$ for the set $\{1,2,\dots,i\}$. For real numbers $x,y$, we write $x\lor y$ to denote $\max\{x,y\}$ and we write $x\land y$ to denote $\min\{x,y\}$. All logs are base $e$.

Finally, we use standard asymptotic notation throughout, as follows. For functions $f=f(n)$ and $g=g(n)$ we write
$f=O(g)$ to mean there is a constant $C$ such that $|f|\le C|g|$, 
we write $f=\Omega(g)$ to mean there is a constant $c>0$ such that $f\ge c|g|$, 
we write $f=\Theta(g)$ to mean that $f=O(g)$ and $f=\Omega(g)$, 
and we write $f=o(g)$ or $g=\omega(f)$ to mean that $f/g\to 0$. 
All asymptotics are taken as $n\to \infty$.

\section{The positive part of \texorpdfstring{\cref{thm:o(n)}}{Theorem~\ref{thm:o(n)}}\label{sec:o(n)}}

In this section we prove that if $k=o\left(n\right)$ then a.a.s.\ $G\left(N\right)\cong\EGs$. The proof consists of two phases. First,
we track the unconstrained evolution of the process until we first
see a matching of size $k-1$. During this time, the $k$-matching-free
process is identical to the basic Erd\H os-R\'enyi random graph
process, and is thus quite easy to analyse. In the second phase, we
begin to track the formation of ``augmenting paths'' that would
allow us to extend a $\left(k-1\right)$-matching into a $k$-matching, and are thus forbidden.
To this end, we will define an evolving partition of the vertex set
into ``components'' of vertices connected by certain special kinds
of paths. We will then couple the $k$-matching-free process with
a much simpler random graph process that captures this component structure,
and study this simpler process via comparison with a certain binomial random
graph.

\subsection{The initial unconstrained evolution\label{subsec:o(n)-initial}}

Let $\nu\left(G\right)$ be the matching number of a graph $G$, and
note that deterministically we have $\nu\left(G\left(t\right)\right)-\nu\left(G\left(t-1\right)\right)\in\left\{ 0,1\right\} $.
So, before the matching number reaches $k-1$, we accept every edge.
Let $\tau=\min\left\{ t:\nu\left(G\left(t\right)\right)=k-1\right\}\ge k-1$
be the time that the matching number reaches $k-1$. In this subsection
we collect some simple a.a.s.\ properties of $\tau$ and $G\left(\tau\right)$.
\begin{lem}
\label{lem:o(n)-tau-bound}A.a.s.\ $\tau\le2k$.
\end{lem}
\begin{proof}
For $t\le 2k$, $G^\all(t-1)$ has at most $2k$ edges (comprising at most $4k$ vertices), so the probability $e(t)$ intersects these edges is at
most $4kn/(\binom{n}2-2k)\le 9k/n$. Therefore, the expected number of steps $t\le 2k$
which do not increase the matching number is at most $2k\left(9k/n\right)=18k^{2}/n=o(k)$.
By Markov's inequality, this number of steps is a.a.s.\ at most $k$,
which proves that a.a.s.\ $\tau\le2k$.
\end{proof}
It follows from \cref{lem:o(n)-tau-bound} that if we can prove that
a decreasing property holds a.a.s.\ for the Erd\H os-R\'enyi random graph $G^{\all}(2k)\in \GG\left(n,2k\right)$, then it holds a.a.s.\ for $G\left(\tau\right)$.
In fact, using say \cite[Proposition~1.15]{JLR}, it suffices to show
that such a property holds a.a.s.\ for the binomial random graph $\GG\left(n,p\right)$,
where $p=2k/N$.
\begin{lem}
\label{lem:o(n)-tau-acylic}A.a.s.\ $G\left(\tau\right)$ is acyclic.
\end{lem}
\begin{proof}
We show that a.a.s.\ $G\in\GG\left(n,p\right)$ is acyclic. Noting
that $np=\Theta\left(k/n\right)=o\left(1\right)$, the expected number
of cycles in $G$ is
\[
\sum_{i=3}^{n}\binom{n}{i}\left(i-1\right)!p^{i}\le\sum_{i=1}^{\infty}\frac{\left(np\right)^{i}}{i}=-\log\left(1-np\right)=o\left(1\right),
\]
and the desired result follows from Markov's inequality.
\end{proof}
Next we show that most components of $G\left(\tau\right)$ are small. Define the \emph{susceptibility} $S(G)$ of a graph $G$ to be the sum of squares of sizes of its components. See for
example \cite{JL} for background on this notion. Let $\tilde{S}\left(G\right)$ be $S(G)$ minus the number of isolated vertices of $G$ (equivalently, $\tilde{S}\left(G\right)$ is the sum of squares
of sizes of nontrivial components of $G$).
\begin{lem}
\label{lem:o(n)-tau-S}A.a.s.\ $\tilde{S}\left(G\left(\tau\right)\right)=o\left(n\right)$
\end{lem}
\begin{proof}
Let $G\in\GG\left(n,p\right)$; we will show that a.a.s.\ $\tilde{S}\left(G\right)=o\left(n\right)$.
Let $X_{v}$ be the size of the component of $v$ in $G$. Conditioning
on the neighbourhood $N_{G}\left(v\right)$ of $v$ in $G$, we have
\[
X_{v}\le 1+\!\!\sum_{w\in N_{G}\left(v\right)}\!\!X_{w}^{v},
\]
where $X_{w}^{v}\le X_{w}$ is the size of the component of $w$ in
$G-v$. Note that $X_{w}^{v}$ does not actually depend on $N_{G}\left(v\right)$,
so $\E\left[X_{w}^{v}\cond N_{G}\left(v\right)\right]=\E X_{w}^{v}\le\E X_{w}$
for all $w\ne v$. Then
\begin{align*}
\E\left[X_{v}\cond N_{G}\left(v\right)\right] &\le 1+\!\!\sum_{w\in N_{G}\left(v\right)}\!\!\E X_{w},\\
\E X_{v} & \le 1+\left(n-1\right)p\,\E X_{v},\\
\left(1-np\right)\E X_{v} & \le 1,\\
\E X_{v} & =1+o\left(1\right).
\end{align*}
Let $Q$ be the number of isolated vertices in $G$, so $\E Q=n(1-p)^{n-1}=ne^{O(np)}=n-o(n)$ and $\E\tilde{S}\left(G\right)=\E\left[\sum_{v}X_{v}\right]-\E Q=o\left(n\right)$. The desired result follows from Markov's inequality.
\end{proof}
\global\long\def\C{\mathcal{C}}
\global\long\def\D{\mathcal{D}}
\global\long\def\F{\mathcal{F}}
\global\long\def\G{\mathcal{G}}
\global\long\def\Q{\mathcal{Q}}
\global\long\def\A{\mathcal{A}}

In view of the above lemmas, for the rest of the proof condition on an outcome of $\tau,e\left(1\right),\dots,e\left(\tau\right)$
such that $\tau\le2k$, and such that $G\left(\tau\right)$ is acyclic
and satisfies $\tilde{S}\left(G\left(\tau\right)\right)=o\left(n\right)$.
Fix a $\left(k-1\right)$-edge matching $M$ in $G\left(\tau\right)$,
let $A$ be its vertex set, and let $B=\range n\setminus A$ contain
the other vertices. For any vertex $a\in A$, let $m_{a}$ be the
unique neighbour of $a$ in $M$. Note that $M$ will be a maximum
matching in $G\left(t\right)$ for each $t\ge\tau$, by the definition of the process. Given our conditioning, note that $e(\tau+1),\dots ,e(N)$ is a uniformly random ordering of the pairs of vertices other than $e(1),\dots,e(\tau)$.

Now, Berge's Lemma \cite[Theorem~1]{Berge} says that a matching is
maximum if and only if there is no \emph{augmenting path}: that is,
a path that starts and ends on unmatched vertices, and alternates
between edges in and not in the matching. This means that each incoming
edge $e\left(t\right)$ will be accepted if and only if its addition
to $G\left(t-1\right)$ does not create an augmenting path with respect
to $M$. For the rest of the paper, ``augmenting path'' will refer
to a path that starts and ends in $B$, and alternates between edges
in $M$ and not in $M$. In order to keep track of the formation of
such alternating paths, we introduce some auxiliary data (``charges''
and ``roots''), which evolve with $G\left(t\right)$, as follows.

\subsection{\label{subsec:o(n)-charges}Charges and roots}

We will define charges $c_{v}\left(t\right)\in \{-1,0,1\}$ and roots $r_{v}\left(t\right)\in \{0\}\cup B$
for each $t\ge\tau$ and each vertex $v$. If the root of a vertex is zero we say it has no root, and if the charge of a vertex is zero we say it is uncharged. To begin with, only the vertices in $B$ will be charged, and as the process $\left(G(t)\right)_t$ evolves, the vertices in $A$ will gradually become charged, gaining root data as this happens (charged vertices will never change their charge or root). The idea is that if a vertex is charged, that means there is an alternating path from that vertex to its root, and the sign of the charge corresponds to the parity of the length of this path. This information will allow us to deduce that certain edges are forbidden by the process.

First, we define ``initial
conditions'', which do not actually correspond to charge and root data at any point of the process, but which will be used as a starting point to define the evolution of the charge and root data. For each $b\in B$, let $c_{b}\left(*\right)=-1$ and
$r_{b}\left(*\right)=b$, meaning that each vertex in $B$ has
negative charge and has itself as a root. For each $a\in A$, let
$c_{a}\left(*\right)=0$ and $r_{a}\left(*\right)=0$, meaning that
each vertex in $A$ has no charge and no root.

Next we describe how the data update at each step. For a graph $G$, and for charge and
root data $\left(c,r\right)$, define $c'\left(G,c,r\right)$ and
$r'\left(G,c,r\right)$ via the following procedure. Start with the
charges and roots given by $c$ and $r$, and repeatedly do the following.
As long as there is an edge in $G$ between a negatively charged vertex
$v$ and an uncharged vertex $a\in A$, give a positive charge to
$a$, give a negative charge to $m_{a}$, and give both of these newly
charged vertices the same root as $v$. (If there are multiple edges
between negatively charged and uncharged vertices, choose the one that was offered first).

Finally, we can define the charge and root data associated with each $G\left(t\right)$,
$t\ge\tau$. Let $c\left(\tau\right)=c'\left(G\left(\tau\right),c\left(*\right),r\left(*\right)\right)$
and $r\left(\tau\right)=r'\left(G\left(\tau\right),c\left(*\right),r\left(*\right)\right)$,
and for $t>\tau$ let $c\left(t\right)=c'\left(G\left(t\right),c\left(t-1\right),r\left(t-1\right)\right)=c'\left(G\left(t\right),c\left(*\right),r\left(*\right)\right)$
and $r\left(t\right)=r'\left(G\left(t\right),c\left(t-1\right),r\left(t-1\right)\right)=r'\left(G\left(t\right),c\left(*\right),r\left(*\right)\right)$. For $t\ge\tau$ and $b\ne 0$ let $C^b\left(t\right)=\left\{ a\in A:r_{a}\left(t\right)=b\right\} $
be the ``charge component'' of vertices in $A$ which have root
$b$, and let $\C\left(t\right)$ be the collection of all such components
which are nonempty. Note that the edges
that were used to charge the vertices of $C^b\left(t\right)$ form
a tree $T^{b}\left(t\right)$ on the vertex set $C^b(t)\cup\{b\}$, rooted at $b$. Also, let $\D\left(t\right)$ be the set of connected
components in the subgraph of $G\left(t\right)$ induced by the $c\left(t\right)$-uncharged
vertices, and define $\F\left(t\right)=\C\left(t\right)\cup\D\left(t\right)$
as the set of ``generalised components'', which partition $A$.
See \cref{fig:charges} for an illustration. 
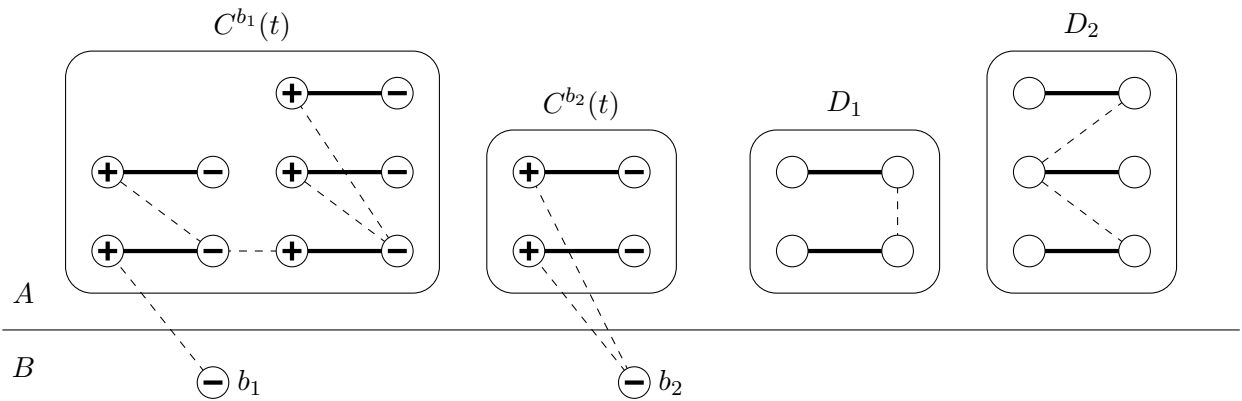
\begin{figure}[h]
\begin{center}
\begin{tikzpicture}[scale=1.4]
\def\minus#1{
 \begin{scope}[shift={#1},scale=0.15]
  \draw [fill=white](0,0) circle (1);
  \draw [ultra thick](-0.6,0) -- (0.6,0);
 \end{scope}
}
\def\plus#1{
 \begin{scope}[shift={#1},scale=0.15]
  \draw [fill=white](0,0) circle (1);
  \draw [ultra thick](-0.6,0) -- (0.6,0);
  \draw [ultra thick](0,-0.6) -- (0,0.6);
 \end{scope}
}
\def\empty#1{
 \begin{scope}[shift={#1},scale=0.15]
  \draw [fill=white](0,0) circle (1);
 \end{scope}
}
\draw(-1,0)--(10.75,0);
\node at (-0.8,0.35){$A$};
\node at (-0.8,-0.35){$B$};

\draw[rounded corners=10] (-0.4,0.35) rectangle (3.15,2.65);
\draw[rounded corners=10] (3.6,0.35) rectangle (5.4,1.9);
\draw[rounded corners=10] (6.1,0.35) rectangle (7.9,1.9);
\draw[rounded corners=10] (8.35,0.35) rectangle (10.15,2.65);

\node (p11) at (0,0.75){};
\node (p12) at (0,1.5){};
\node (p13) at (1.75,2.25){};
\node (n11) at (1,0.75){};
\node (n12) at (1,1.5){};
\node (n13) at (2.75,2.25){};
\node (n10) at (1,-0.5){};
\node at (1.35,-0.5){$b_1$};
\node at (1.375,2.9){$C^{b_1}(t)$};

\node (p21) at (1.75,0.75){};
\node (p22) at (1.75,1.5){};
\node (n21) at (2.75,0.75){};
\node (n22) at (2.75,1.5){};

\node (p31) at (4,0.75){};
\node (p32) at (4,1.5){};
\node (n31) at (5,0.75){};
\node (n32) at (5,1.5){};
\node (n30) at (5,-0.5){};
\node at (5.35,-0.5){$b_2$};
\node at (4.5,2.15){$C^{b_2}(t)$};

\node (p41) at (6.5,0.75){};
\node (p42) at (6.5,1.5){};
\node (n41) at (7.5,0.75){};
\node (n42) at (7.5,1.5){};
\node at (7,2.15){$D_1$};

\node (p51) at (8.75,0.75){};
\node (p52) at (8.75,1.5){};
\node (n51) at (9.75,0.75){};
\node (n52) at (9.75,1.5){};
\node (p53) at (8.75,2.25){};
\node (n53) at (9.75,2.25){};
\node at (9.25,2.9){$D_2$};

\draw[ultra thick] (p11)--(n11);
\draw[ultra thick] (p12)--(n12);
\draw[ultra thick] (p13)--(n13);
\draw[ultra thick] (p21)--(n21);
\draw[ultra thick] (p22)--(n22);
\draw[ultra thick] (p31)--(n31);
\draw[ultra thick] (p32)--(n32);
\draw[ultra thick] (p41)--(n41);
\draw[ultra thick] (p42)--(n42);
\draw[ultra thick] (p42)--(n42);
\draw[ultra thick] (p51)--(n51);
\draw[ultra thick] (p52)--(n52);
\draw[ultra thick] (p53)--(n53);

\draw[dashed] (n11) -- (p12);
\draw[dashed] (n21) -- (p13);
\draw[dashed] (n11) -- (p21);
\draw[dashed] (n21) -- (p22);
\draw[dashed] (n41) -- (n42);
\draw[dashed] (n10) -- (p11);
\draw[dashed] (n30) -- (p31);
\draw[dashed] (n30) -- (p32);
\draw[dashed] (n51) -- (p52);
\draw[dashed] (n53) -- (p52);

\plus{(p11)}\plus{(p12)}\plus{(p13)}
\plus{(p21)}\plus{(p22)}
\plus{(p31)}\plus{(p32)}
\minus{(n11)}\minus{(n12)}\minus{(n13)}\minus{(n10)}
\minus{(n21)}\minus{(n22)}
\minus{(n31)}\minus{(n32)}\minus{(n30)}
\empty{(p41)}\empty{(p42)}\empty{(n41)}\empty{(n42)}
\empty{(p51)}\empty{(p52)}\empty{(p53)}\empty{(n51)}\empty{(n52)}\empty{(n53)}

\end{tikzpicture}
\end{center}

\caption{\label{fig:charges}An example of the state of the charge and root
data at some time $t\ge\tau$. The solid edges are edges of $M$,
and $\protect\D\left(t\right)=\left\{ D_{1},D_{2}\right\} $. Only the edges in the trees $T^b(t)$, and the edges in the uncharged components, are depicted.}
\end{figure}

We will next show that to prove \cref{thm:o(n)} it suffices, roughly
speaking, to prove that edges within generalised components are much
rarer than edges between $A$ and $B$. To state this as a lemma,
we define some hitting times, for each $a\in A$. (Formally, we allow these
hitting times to take the value $\infty$ if their corresponding events
never occur).

\begin{itemize}
\item Let $\tau_{a}^{F}$ be the first time $t>\tau$ that we are offered
an edge $e\left(t\right)$ between $a$ and the rest of its generalised
component in $\F\left(t-1\right)$, or between $a$ and $r_a(t-1)$.
\item Let $\tau_{a}^{B}$ be the first time $t>\tau$ that we are offered
an edge $e\left(t\right)$ between $a$ and $B\setminus\left\{ r_{a}\left(t-1\right)\right\} $.
Note that $B\setminus\left\{ r_{a}\left(t-1\right)\right\} =B$ if
$a$ is uncharged at time $t-1$. Note also that $\tau_{a}^{B}<\infty$
because we are assuming that $\tau<2k<\left|B\right|$.
\item Let $\tau_{a}^{C}\le \tau_{a}^{B}$ be the time $t\ge\tau$ at which
$a$ becomes charged.
\end{itemize}
\begin{lem}
\label{lem:o(n)-main-reduction}If $\tau_{v}^{F}>\tau_{m_{v}}^{B}$
for all $v\in A$, then $G\left(N\right)\cong\EGs$.
\end{lem}
\begin{proof}
We will show that if $\tau_{v}^{F}>\tau_{m_{v}}^{B}$ for all $v$,
then $G\left(N\right)$ has no edges between negatively charged vertices,
which implies that $G\left(N\right)\cong\EGs$. Indeed, at the end of the process each edge of $M$ will have one positively charged and one negatively charged vertex, so there will be
$k-1$ negatively charged vertices in $A$. Combined with the $n-2(k-1)$ negatively charged vertices in $B$, we will have proved that $G\left(N\right)$ has an independent set
of size $n-k+1$, which means it is isomorphic to a subgraph of $\EGs$.
But $G\left(N\right)$ is $k$-matching-saturated, so it cannot be a proper subgraph of the $k$-matching-free graph $\EGs$.

Note first that there can never be any edge between negatively charged
vertices with different roots $b,b'\in B$, because this would give
an augmenting path between $b$ and $b'$. Now, we consider the possible ways that an edge between negatively charged vertices with the same root could arise. The simplest possibility is that we could accept an edge $e(t)$ between two such vertices that are already negatively charged. The second possibility is that the process of charging
vertices (via the introduction of an edge $e(t)$ between a negatively charged vertex $v$ with root $b$ and an uncharged vertex $a$) can somehow result in the previously uncharged endpoints of an existing edge $e$ becoming negatively charged. Observe that this second possibility can only occur if $e$ was previously in a cycle in its uncharged component. Indeed, the entire subtree $T$ of $T^{b}\left(t\right)$ rooted at $a$ would have been newly charged at step $t$, and since the charges $c\left(t\right)$
give a proper 2-colouring of $T$, $T\cup \{e\}$ must have had a cycle. Since we are assuming that $G\left(\tau\right)$
is acyclic, it suffices to prove:
\begin{itemize}
\item [(1)]we never accept an edge that creates a cycle among the uncharged
vertices, and;
\item [(2)]we never accept an edge between two negatively charged vertices with the same root.
\end{itemize}
First, since $\tau_{v}^{F}>\tau_{m_{v}}^{B}\ge\tau_{v}^{C}$ for each
$v$, we are never even offered an edge $e\left(t\right)$ between an uncharged
vertex and the rest of its component $D\in\D\left(t-1\right)$. This
immediately proves (1).

Next, consider a vertex $v$ which becomes negatively charged (with
root $b$) at time $\tau_{v}^{C}$. Let $S_{v}$ be the set of negatively
charged vertices $w\in C^b\left(\tau_{v}^{C}\right)\cup \{b\}$ such that
the unique path between $b$ and $w$ in the tree $T^{b}\left(\tau_{v}^{C}\right)$
does not pass through $v$. (This set $S_{v}$ does not evolve with
$t$). Note that at any time $t$, for any negatively charged distinct $v,w\in C^b\left(t\right)\cup \{b\}$,
we always have $w\in S_{v}$ or $v\in S_{w}$ (in particular, we will
have $w\in S_{v}$ if $v$ was charged later than $w$). The relevance
of these sets is that if there is already an edge from $m_{v}$ to $B\setminus\left\{ b\right\} $,
then an edge from $v$ to $S_{v}$ would create an augmenting path,
so is forbidden.

\mk{would a picture be helpful here?}

For any $v\in A$, note that if conditions (1) and (2) have not been
violated yet at time $\tau_{m_{v}}^{B}-1$, then $G\left(\tau_{m_{v}}^{B}-1\right)$
has an independent set of size $n-k+1$, consisting of the negatively
charged vertices and one colour class (not containing $m_v$) of a 2-colouring of the uncharged
vertices. Unless $m_{v}$ is already negatively charged, meaning that $v$ is positively charged, this independent set would not be affected by the addition
of $e\left(\tau_{m_{v}}^{B}\right)$, so $G(t-1)+e(t)$ has no $k$-matching and $e\left(\tau_{m_{v}}^{B}\right)$ is accepted. Since $\tau_{m_{v}}^{B}<\tau_{v}^{F}$, this means that if $v$ is ever negatively charged then no edge between $v$ and $S_{v}$ can ever
be accepted. Applying this argument iteratively to all $v$, in order
of $\tau_{m_{v}}^{B}$, proves (2).
\end{proof}

Now, to prove that edges within generalised components are much rarer than edges between $A$ and $B$, it will suffice to show that most generalised components are likely to remain ``small'' throughout the process. For a partition $\G$ of $A$, let $S(\G)$ be the sum of squares of sizes of its parts. It would be most natural to try to show that $S(\F(t))$ is small for each $t$, but for technical reasons it is more convenient to individually deal with the $\C(t)$ and $\D(t)$. We can view each $\C(t)$ (respectively, each $\D(t)$) as a partition of $A$ by putting each uncharged (respectively, charged) vertex in its own singleton part. Note that the sequence of partitions $\C(t)$ is ``monotone'' in the sense that for each $t>\tau$, $\C(t-1)$ is a refinement of $\C(t)$. This is not true for the $\D(t)$, because when part of an uncharged component gains charge, it splits into singleton components. So, let $\overline\D(t)$ be the finest common coarsening of the partitions $\D(\tau),\dots,\D(t)$. Equivalently, this means that $\overline\D(t)$ is the set of connected components of the union of all the uncharged subgraphs up to time $t$. The following lemma reduces \cref{thm:o(n)} to a.a.s.\ bounds on $S\left(\C(N)\right)$ and $S\left(\overline\D(N)\right)$.
\begin{lem}
\label{lem:o(n)-simplified-process-reduction}To prove that a.a.s.\ $\tau_{v}^{F}>\tau_{m_{v}}^{B}$ for all $v\in A$, it suffices to
prove that a.a.s.\ $$S\left(\C(N)\right),\;S(\overline\D(N))=o(n).$$
\end{lem}
\begin{proof}
For each $v\in A$, let $F_{v}\left(t\right)\in\F\left(t\right)$, $C_v(t)\in\C\left(t\right)$ and $D_{v}\left(t\right)\in\overline\D\left(t\right)$ be the parts containing $v$ in the partitions $\F\left(t\right)$, $\C\left(t\right)$ and $\overline\D\left(t\right)$ respectively. Let $X_v(t)=|C_v(t)|+|D_v(t)|$, and let $S^X(t)=\sum_{v\in A}X_v(t)$. So, $|F_v(t)|\le X_v(t)$, the sequence of $S^X(t)$ is monotone nondecreasing in $t$, and we are assuming that a.a.s.\ $S^X(N)=S(\C(N))+S(\overline\D(N))=o(n)$.

By considering the events $
\{t=\tau_{v}^{F}<\tau_{m_{v}}^{B}\}$ for
each $t\ge\tau$ and $v\in A$, it would be possible to show that
\[
\Pr\left(\bigcup_{v\in A}\left\{ \tau_{v}^{F}<\tau_{m_{v}}^{B}\right\} \right)\le\frac{2\sum_{v\in A}\E X_v(N)}{n}=\frac{2\,\E S^X\left(N\right)}{n}.
\]
However, an a.a.s.\ bound on $S^X(N)$ does
not (directly) imply that $\E S^X(N)=o\left(n\right)$.
\mk{This is very annoying, am I missing something obvious here?}
To overcome this difficulty, we essentially stop the process as soon
as $S^X(N)$ gets too large. To be precise,
choose $f=o\left(n\right)$ such that a.a.s.\ $S^X(N)\le f$,
and let 
\[
\tau_{f}^{S}=\min\left\{ t\le N:S^X(t)>f\right\} \land\left(N+1\right),
\]
so that a.a.s.\ $\tau_{f}^{S}=N+1$, and therefore a.a.s.\ $\tau_{m_{v}}^{B}\land\tau_{f}^{S}=\tau_{m_{v}}^{B}$. Then, $S^X(\tau_{f}^{S}-1)\le f=o\left(n\right)$,
so it suffices to show that
\begin{equation}
\Pr\left(\bigcup_{v\in A}\left\{ \tau_{v}^{F}\le\tau_{m_{v}}^{B}\land\tau_{f}^{S}\right\} \right)\le\frac{2\,\E S^X(\tau_{f}^{S}-1)}{n}.\label{eq:o(n)-desired-stopped-prob}
\end{equation}
Consider any $t>\tau$ and $v\in A$, and condition on $e\left(\tau+1\right),\dots,e\left(t-1\right)$.
If $t-1<\tau_{v}^{F}\land\tau_{m_{v}}^{B}\land\tau_{f}^{S}$ then
there are at most $\left|(F_{v}\left(t-1\right)\cup\{r_v(t-1)\})\setminus\{v\}\right|\le \left|F_{v}\left(t-1\right)\right|$
choices for $e\left(t\right)$ that would cause $t=\tau_{v}^{F}$. There are at least $\left|B\setminus\left\{ r_{v}\left(t-1\right)\right\} \right|-\tau\ge n/2$
choices for $e\left(t\right)$ that would cause $t=\tau_{v}^{F}\land\tau_{m_{v}}^{B}$. (Here we account for the fact that $\tau$ edges had already been offered at time $\tau$, and are therefore not viable candidates for $e(t)$).
So,
\begin{align*}
&\Pr\left(\tau_{v}^{F}=t\le\tau_{f}^{S}\cond\tau_{v}^{F}\land\tau_{m_{v}}^{B}=t,\,e\left(\tau+1\right),\dots,e\left(t-1\right)\right)\\
&\qquad=\frac{
\Pr\left(
\tau_{v}^{F}=t\le\tau_{f}^{S}\cond e\left(\tau+1\right),\dots,e\left(t-1\right)\right)
}
{
\Pr\left(\vphantom{\tau_{f}^{S}}\tau_{v}^{F}\land\tau_{m_{v}}=t\cond e\left(\tau+1\right),\dots,e\left(t-1\right)\right)
}
\\
&\qquad \le
\frac{\left|F_{v}\left(t-1\right)\right|}{n/2}.
\end{align*}
It follows that
\begin{align*}
\Pr\left(\tau_{v}^{F}\le\tau_{m_{v}}^{B}\land\tau_{f}^{S}\cond\tau_{v}^{F}\land\tau_{m_{v}}^{B}=t\right) & =\Pr\left(\tau_{v}^{F}=t\le\tau_{f}^{S}\cond\tau_{v}^{F}\land\tau_{m_{v}}^{B}=t\right)\\
 & \le\frac{2\,\E
\left|F_{v}\left(\tau_{f}^{S}-1\right)\right|
}{n}
  \le\frac{2\,\E X_v(\tau_{f}^{S}-1)}{n}.
\end{align*}
\mk{It seems like what I'm doing here resembles the optional stopping
theorem. Would it be better to write this in a way that actually uses
the optional stopping theorem?}
Since this holds for all $t$, we
in fact have
\[
\Pr\left(\tau_{v}^{F}\le\tau_{m_{v}}^{B}\land\tau_{f}^{S}\right)\le\frac{2\,\E X_v(\tau_{f}^{S}-1)}{n}.
\]
The desired inequality \cref{eq:o(n)-desired-stopped-prob} follows, by the union bound.
\end{proof}

\subsection{\label{subsec:o(n)-coupling}Coupling with a simpler process}

In this section we define an auxiliary random graph process $G'\left(t\right)$ based
on $G\left(t\right)$, which captures most of its generalised component structure but is much simpler to analyse. Each $G'\left(t\right)$
will be a graph on the vertex set $A$. To start with, define $G'\left(\tau\right)$ to contain a clique on the vertex set of each generalised component $F\in \F(\tau)$. Then, for every $t>\tau$ and every $v\in A$,
let $E_{v}\left(t\right)$ be the event that there was an edge $e\left(t'\right)$,
with $\tau\le t'\le t$, between $B$ and $v$. Let $e\left(t\right)=\left\{ v\left(t\right),w\left(t\right)\right\} $,
and for all $t>\tau$, let 
\[
G'\left(t\right)=\begin{cases}
G'\left(t-1\right) & \text{if }E_{v\left(t\right)}\left(t\right)\text{ and }E_{w\left(t\right)}\left(t\right)\text{ both hold, or if }e\left(t\right)\nsubseteq A;\\
G'\left(t-1\right)+e\left(t\right) & \text{otherwise.}
\end{cases}
\]
That is to say, we reject an edge within $A$ only if both of its endpoints have already
been offered an edge to $B$.

Now, let $\F'(t)$ be the set of connected components of $G'(t)$. We would like to be able to say that for each $t$, the partition $\F\left(t\right)$ is a refinement of the partition $\F'(t)$, so that we can control $S(\C(N))$ and $S(\overline \D(N))$ via $S(\F'(N))$. This turns out to be almost true, except for the fact that an uncharged component $D$ can merge with a charge component $C^b(t-1)$ via an edge $e(t)$ from $D$ to $b$, and such edges are ``invisible'' to the process $G'(t)$.

So, we define a slight refinement $\C'(t)$ of $\C(t)$ ignoring these ``invisible'' edges, such that $\C'\left(t\right)$ really is a refinement of $\F'(t)$; we will bound $S(\C'(N))$ via $S(\F'(N))$ and deal with the coarsening of $\C'(t)$ to $\C(t)$ separately. First we define the refinement $\C_\sub(t)$ of $\C(t)$ obtained by splitting each $C^b(t)$ into the connected components of the forest $T^b(t)-b$ (we call these ``sub-components''). Then let $\C'(t)$ be the finest common coarsening of $\C_\sub(t)$ and $\C(\tau)$. That is to say, we only group vertices which are in the same charge component because they were directly charged by each other, not by their external root, unless they had already been charged in this way by time $\tau$.

\begin{lem}
\label{lem:coarser}$\C'\left(N\right)$ and $\overline{\D}(N)$ are both refinements of $\F'(N)$, as partitions of $A$.
\end{lem}
\begin{proof}
Let $\G(t)=\D(t)\cup \C'(t)$. Since $\F'(t)$ is monotone in the sense that each $\F'(t-1)$ is a refinement of $\F(t)$, it suffices to prove that for each $\tau \le t\le N$, $\G(t)$ is a refinement of $\F'(t)$. First note that $\G(\tau)=\F'(\tau)$ by definition. Now, there are two ways $\G(t)$ can differ from $\G(t-1)$. The first possibility is that an edge $e(t)$ is accepted for $G(t)$ between two uncharged vertices in different components $D_1,D_2\in \D(t-1)$, in which case those components are merged to give $\D(t)$. The second possibility is that $e(t)$ is accepted for $G(t)$ between an uncharged vertex (in a component $D\in \D(t)$, say) and a negatively charged vertex $w$ (with root $b$, say), in which case some subset $U\subseteq D$ gains charge and is added to the relevant sub-component of $C^b\left(t-1\right)$. In this case, $w$ could be $b$ itself (in which case $\G(t)$ is a refinement of $\G(t-1)$), or $w$ could be in $C^b(t-1)\subseteq A$. Considering all possibilities, it suffices to show that every edge between an uncharged vertex and another vertex in $A$, which is accepted for $G(t)$, is also accepted for $G'(t)$. To see this, note that if $E_{v}\left(t\right)$ holds, then $t\ge\tau_{v}^{B}\ge\tau_{v}^{C}$, meaning that $v$ is charged at time $t$.
\end{proof}

Now we show how to control $S\left(\C(N)\right)$ and $S\left(\overline\D(N)\right)$ via $S(\F'(N))=S(G'(N))$.
\begin{lem}
\label{lem:o(n)-eliminate-fix}To prove that a.a.s.\ $S\left(\C(N)\right), S\left(\overline\D(N)\right)=o\left(n\right)$, it suffices to prove that a.a.s.\ $S\left(G'\left(N\right)\right)=o\left(n\right)$.
\end{lem}
\begin{proof}
First note that by \cref{lem:coarser}, a.a.s.\ $S(\overline\D(N))\le S(\F'(N))=o(n)$.

Next, let $\mathcal E$ be the relative ordering of all edges in the sequence $e(\tau+1),\dots,e(N)$ which are between $A$ and $B$. Recalling that we are conditioning on $e(1),\dots,e(\tau)$, and observing that $\C'(t)$ can only differ from $\C'(t-1)$ if $e(t)$ is between two vertices in $A$, note that $\C'(N)$ does not depend on $\mathcal E$. By \cref{lem:coarser}, we can assume that a.a.s.\ $S(\C'(N))\le S(\F'(N))=o\left(n\right)$, so condition on an outcome of $\C'(N)$ with this property. This does not change the distribution of $\mathcal E$.

Now, for $C\in \C'(N)$, let $r(C)$ be the common root of all elements of $C$, corresponding to the first edge that was offered between $C$ and $B$. For some $C\in \C'(N)$, we may have already seen an edge between $C$ and $B$ by time $\tau$, meaning that $r(C)$ is determined. Let $\Q$ be the set of such $C$. For all other $C$, by the randomness of $\mathcal E$, the root $r(C)$ is uniformly distributed in $B$, and these roots are independent of each other.

For each $C\in\C'(N)$, let $C^*\in \C(N)$ be charge component which includes $C$. We will estimate each $\E\,|C^*|$. To this end, note that
$$C^*=\!\!\bigcup_{\substack{\vspace{2px} C'\in \C'(N):\\r(C')=r(C)}}\!\!C'.$$
Now, in in the case where $C\notin \Q$, for each $C'\ne C$ we have $r(C)=r(C')$ with probability $1/|B|=(1+o(1))/n$, so
$$\E\,|C^*|\le|C|+(1+o(1))\sum_{C'\in \C'(N)}\! \frac{|C'|}n=|C|+o(1).$$
Alternatively, if $C\in \Q$ then $C$ is the only component in $\Q$ with root $r(C)$, and for each $C'\notin \Q$ we have $r(C)=r(C')$ with probability $1/|B|$, so
\begin{align*}
\E\,|C^*|&\le|C|+(1+o(1))\!\!\sum_{C'\in \C'(N)\setminus\Q}\! \frac{|C'|}n=|C|+o(1).
\end{align*}
Combining both cases, with $C_v\in\C'\left(N\right)$ as the part containing $v$ in $\C'\left(N\right)$, we have
$$\E\,S(\C(N))= \sum_{v\in A}\E\,|C_v^*|\le \sum_{v\in A}|C_v|+o\left(|A|\right)=O\left(S(\C'(N))\right)=o(n).$$
The desired result follows by Markov's inequality.
\end{proof}

Now, apart from the edges in $G'\left(\tau\right)$, note that each
edge in $G'\left(N\right)$ is present with probability $\Theta(1/n)$.
This is because the condition for including an edge $\{u,v\}$ is that it must be offered before any of the edges between $u$ and $B$, or before any of the edges between $v$ and $B$, and there are $\Theta(n)$ edges of both types. Although the edges of $G'(N)$ are not independent, we will see in the next subsection that we can nevertheless reduce our problem to a comparable problem concerning a certain binomial random graph where each edge is independently present with probability $1/n$. We remark that the susceptibility of the union of a random graph and a fixed sparse random-like graph has already been studied, by Spencer and Wormald \cite{SW} and by Bohman, Frieze, Krivelevich,
Loh and Sudakov \cite{BFKLS}. However our methods will be much simpler, and will resemble the proof of \cref{lem:o(n)-tau-S}.

\subsection{\label{subsec:o(n)-binomial}Reducing to a binomial random graph}

First we define intermediate random graphs $G^{\Unif}$, $G^{\Exp}$
on the vertex set $A$, which stochastically dominate $G'\left(N\right)$. Let $B'\subseteq B$ be a set of $n-o(n)$ isolated vertices in $G(\tau)$, let $\left(\eta_{a,b}\right)_{a\in A,\,b\in B'}$ and $\left(\eta_{e}\right)_{e\in\binom{A}{2}}$
be independent random variables uniformly distributed in the interval
$\left[0,1\right]$. Almost surely each is distinct, so these random
variables induce a uniformly random ordering of the possible edges
within $A$ and between $A$ and $B$. Put an edge $e=\left\{ v,w\right\} $
in $G^{\Unif}$ if $e\in G'\left(\tau\right)$, or if $\eta_{e}\le\left(\min_{b\in B'}\eta_{v,b}\right)\lor\left(\min_{b\in B'}\eta_{w,b}\right)$.
This graph is defined so that it stochastically dominates $G'\left(N\right)$.

Now, note that the uniform distribution $\Unif\left(0,1\right)$ is
stochastically dominated by the exponential distribution $\Exp\left(1\right)$
(one can see this by comparing cumulative distribution functions),
and recall that the minimum of $m$ independent $\Exp\left(1\right)$
random variables has the distribution $\Exp\left(m\right)$. So, let
$\left(\gamma_{v}\right)_{v\in A}$ be independent $\Exp\left(\left|B'\right|\right)$
random variables, and define $G^{\Exp}$ by putting an edge $e=\left\{ v,w\right\} $
in $G^{\Exp}$ if $e\in G'\left(\tau\right)$, or if $\eta_{e}\le\gamma_{v}\lor\gamma_{w}$.
Then $G^{\Exp}$ stochastically dominates $G^{\Unif}$.

\global\long\def\Gp{G^p}

Next, let $G^{*}$ be the graph obtained by starting with $G'\left(\tau\right)$
and blowing up each vertex $v$ into a clique of size $2\floor{\gamma_{v}n}+3$.
This means that each vertex is replaced with a clique, and two vertices
of $G^{*}$ in different cliques are adjacent if their corresponding
vertices in $G'\left(\tau\right)$ were adjacent. Given $G^{*}$ (with vertex set $V^*$, say), let $\Gp$ be a random graph on the same vertex set $V^*$, where each of the $\binom{|V^*|}2$ edges is independently present with probability $1/n$. We will
show that in a certain sense $G^{*}\cup \Gp$ dominates $G^{\Exp}$.
\begin{lem}
\label{lem:blowup-dominates}Let $E$ be the event that each $\gamma_{v}\le2\log n/n$,
and condition on $\left(\gamma_{v}\right)_{v\in A}$ satisfying $E$.
For each $v,w\in A$ with $v\ne w$, let $p_{v,w}$ be the probability
that there is at least one edge between the blowup of $v$ and the
blowup of $w$, in $G$. Then $p_{v,w}\ge\gamma_{v}\lor\gamma_{w}$.
\end{lem}
\begin{proof}
We have
\[
p_{v,w}:=1-\left(1-1/n\right)^{\left(2\floor{\gamma_{v}n}+3\right)\left(2\floor{\gamma_{w}n}+3\right)}.
\]
By Taylor's theorem (expanding around $p=0$), for all $p,x\ge 0$,
\[
\left(1-p\right)^{x}\le1-px+p^{2}x(x-1)/2\le 1-px\left(1-px\right),
\]
and since each $\gamma_{v}\le2\log n/n$, we have $\left(2\floor{\gamma_{v}n}+3\right)\left(2\floor{\gamma_{w}n}+3\right)\le17\log^{2}n$.
So,
\begin{align*}
p_{v,w} & \ge\left(1-17\log^{2}n/n\right)\left(2\floor{\gamma_{v}n}+3\right)\left(2\floor{\gamma_{w}n}+3\right)/n\\
 & >\frac{1}{2}\left(2\gamma_{v}n+1\right)\left(2\gamma_{w}n+1\right)/n\\
 & \ge\gamma_{v}+\gamma_{w} \ge\gamma_{v}\lor\gamma_{w},
\end{align*}
as desired.
\end{proof}
\cref{lem:blowup-dominates} shows that $G^{\Exp}$ and $G^{*}\cup \Gp$
can be coupled in such a way that $S\left(G^{*}\cup \Gp\right)\ge S\left(G^{\Exp}\right)$
whenever $E$ holds. Note that $\Pr\left(\gamma_{v}>x\right)=e^{-\left|B'\right|x}$,
so $\Pr\left(\gamma_{v}>2\log n/n\right)=o\left(1/n\right)$, and
in fact $E$ a.a.s.\ holds. Recalling that $G^{\Exp}$ stochastically
dominates $G'\left(N\right)$, to prove \cref{thm:o(n)} via \cref{lem:o(n)-main-reduction,lem:o(n)-simplified-process-reduction,lem:o(n)-eliminate-fix}
it suffices to prove the following lemma.
\begin{lem}
\label{lem:o(n)-S}A.a.s.\ $S\left(G^{*}\cup \Gp\right)=o\left(n\right)$.
\end{lem}
\begin{proof}
For each connected component $F\in \F'(\tau)$, let $X_F=\sum_{v\in F}\gamma_v$.
 The size of the corresponding blown-up component in $G^{*}$ is $3|F|+2n\sum_{v\in F} \floor{\gamma_v}\le 3|F|+2nX_F$. Now, $\E\gamma_v=1/|B'|$ and $\Var\gamma_v=1/|B'|^2$, so $\E X_{F}=\left|F\right|/\left|B'\right|$, $\Var X_{F}=\left|F\right|/\left|B'\right|^{2}$ and $\E X_F^{2}=\left|F\right|\left(\left|F\right|+1\right)/\left|B'\right|^{2}$.
Recalling that $\left|B'\right|=n-o\left(n\right)$, we have 
\[
\E\left(3\left|F\right|+2nX_{F}\right)^{2}=9\left|F\right|^{2}+12n\left|F\right|\E X_{F}+4n^{2}\E X_{F}^{2}=O\left(\left|F\right|^{2}\right).
\]
The components of $G'\left(\tau\right)$ are subsets of nontrivial
components of $G\left(\tau\right)$, so $S\left(G'\left(\tau\right)\right)\le\tilde{S}\left(G\left(\tau\right)\right)=o\left(n\right)$,
and
\[
\E S\left(G^{*}\right)=\sum_{F\in\F'(\tau)}\E\left(3\left|F\right|+2nX_{F}\right)^{2}=O\left(\sum_{F\in\F'(\tau)}\left|F\right|^{2}\right)=O\left(S\left(G'\left(\tau\right)\right)\right)=o\left(n\right).
\]
By Markov's inequality, a.a.s.\ $S\left(G^{*}\right)=o\left(n\right)$,
so for the rest of the proof we condition on an outcome of $G^{*}$
satisfying this property.

Now we proceed in a similar way to the proof of \cref{lem:o(n)-tau-S}.
Let $F_{v}$ be the component of $v$ in $G^{*}$, let $G_{v}^{*}$
be obtained from $G^{*}$ by deleting all edges incident to $v$,
and let $Y_{v}$ be the size of the component of $v$ in $G_{v}^{*}\cup \Gp$.
Let $Y_{w}^{v}\le Y_{w}$ be the size of the component of $w$ in
$\left(G_{w}^{*}\cup \Gp\right)-v$. Conditioning on $N_{\Gp}\left(v\right)$,
we have
\[
Y_{v}\le1+\!\!\sum_{u\in N_{\Gp}\left(v\right)}\;\sum_{w\in F_{u}}Y_{w}^{v}.
\]
Note that $Y_{w}^{v}$ does not depend on $N_{\Gp}\left(v\right)$,
so $\E\left[Y_{w}^{v}\cond N_{\Gp}\left(v\right)\right]=\E Y_{w}^{v}\le\E Y_{w}$
for all $w\ne v$. If $v$ is chosen to maximise $\E Y_{v}$, we have
\begin{align*}
\E\left[Y_{v}\cond N_{\Gp}\left(v\right)\right] & \le1+\!\!\sum_{u\in N_{\Gp}\left(v\right)}\;\sum_{w\in F_{u}}\E Y_{w}\\
& \le1+\E Y_v\!\!\sum_{u\in N_{\Gp}\left(v\right)}\!\!|F_u|,\\
\E Y_{v} & \le1+\left(S\left(G^{*}\right)/n\right)\,\E Y_{v},\\
\left(1-o\left(1\right)\right)\,\E Y_{v} & \le1,\\
\E Y_{v} & =1+o\left(1\right).
\end{align*}
Then, observe that the size of the component of $w$ in $G^{*}\cup \Gp$
is at most $\sum_{u\in F_{w}}Y_{u}$, so its expected size is $(1+o(1))|F_w|$ and
\[
\E S\left(G^{*}\cup \Gp\right)=\left(1+o\left(1\right)\right)S\left(G^{*}\right)=o\left(n\right).
\]
The desired result follows by Markov's inequality.
\end{proof}

\section{\label{sec:negative-o(n)}The negative part of \texorpdfstring{\cref{thm:o(n)}}{Theorem~\ref{thm:o(n)}}}

In this section we prove that, for any fixed $\varepsilon>0$, if $\varepsilon n\le k\le\left(1-\varepsilon\right)n/2$,
then $G\left(N\right)\ncong\EGs$ with probability $\Omega(1)$. This suffices, because the case where $n-2k=o(n)$ is handled by \cref{thm:not-extremal}. Define $\tau$ as in \cref{sec:o(n)}.
\begin{lem}
\label{lem:n-tau-bound}There is $R=R\left(\varepsilon\right)$ such
that a.a.s.\ $\tau\le Rn$.
\end{lem}
\begin{proof}
We prove that in the Erd\H os-R\'enyi random graph $G^{\all}(Rn)\in \GG(n,Rn)$ there is a.a.s.\ a $k$-matching. Using \cite[Proposition~1.15]{JLR}, it actually suffices to show this for a
binomial random graph $G\in\GG\left(n,R/n\right)$. To do this, we prove that there is no independent
set of size $\varepsilon n$. Indeed, for any set of $\varepsilon n$
vertices, the probability that the set is independent is

\[
\left(1-R/n\right)^{{\varepsilon n \choose 2}}\le e^{-\Omega\left(\varepsilon^{2}Rn\right)}.
\]
If $R$ is much larger than $\varepsilon^{-2}$, then this probability
is $o\left(2^{-n}\right)$, so the union bound says that a.a.s.\ $G$
has no independent set of size $\varepsilon n$, as desired.
\end{proof}
Now, say a triangle in a graph is \emph{isolated }if there are no
edges between the triangle and the rest of the graph.
\begin{lem}
\label{lem:isolated-triangle}$G\left(\tau\right)$ has an isolated
triangle with probability $\Omega\left(1\right)$.
\end{lem}
\begin{proof}
Let $R$ be as in \cref{lem:n-tau-bound}, and consider independent
binomial random graphs $G_1\in\GG\left(n,p_1\right)$ and $G_2\in\GG\left(n,p_2\right)$,
with $p_1=\varepsilon/n$ and $p_2=3R/n$. For any $m$, the distribution of $G_1$ (respectively $G_1\cup G_2$) conditioned on the event that $e(G_1)=m$ (respectively, that $e(G_1\cup G_2)=m$), is precisely the Erd\H os-R\'enyi distribution $\GG(n,m)$. Note that a.a.s.\ $G_1$ has
fewer than $k-1\le \tau$ edges, and a.a.s.\ $G_2$ has at least $Rn$ edges.
So, $G_1$, $G_2$ and $\left(G\left(t\right)\right)_{t}$ can be coupled
together such that  a.a.s.\ $G_1=G^{\all}(e(G_1))\subseteq G\left(\tau\right)$ and $G_1\cup G_2=G^{\all}(e(G_1\cup G_2))\supseteq G\left(\tau\right)$.
We will prove that $G_1$ has an isolated triangle with probability
$\Omega\left(1\right)$, and we will then show that conditioning on
an outcome of $G_1$ with an isolated triangle, that triangle remains
isolated in $G_1\cup G_2$ with probability $\Omega\left(1\right)$.

Let $X$ be the number of isolated triangles in $G_1$, so that
\[
\E X={n \choose 3}p_1^{3}\left(1-p_1\right)^{3\left(n-3\right)}=\left(1+o\left(1\right)\right)\frac{\varepsilon^{3}}{6}e^{-3\varepsilon}=\Theta\left(1\right),
\]
and similarly
\begin{align*}
\E\left[X\left(X-1\right)\right] & ={n \choose 3}{n-3 \choose 3}p_1^{6}\left(1-p_1\right)^{6\left(n-3\right)+9}=\Theta\left(1\right),\\
\E X^{2} & =\E\left[X\left(X-1\right)\right]+\E X=\Theta\left(1\right).
\end{align*}
So, $\left(\E X\right)^{2}/\E X^{2}  =\Omega\left(1\right)$, 
and by the second moment method, $G_1$ has an isolated triangle with
probability $\Omega\left(1\right)$. Now, condition on an outcome
of $G_1$ with an isolated triangle $T$. The probability $T$ is still
isolated in $G_1\cup G_2$ is $\left(1-p_2\right)^{3n-9}=\Omega\left(1\right)$.
\end{proof}
Now, if $G\left(\tau\right)$ has an isolated triangle, then $G\left(N\right)\ncong\EGs$.
Indeed, fix an isolated triangle $T$ and a maximum $\left(k-1\right)$-matching
$M$ in $G\left(\tau\right)$. Observe that exactly one edge of $T$
is used in $M$. We can never accept an edge between $T$ and the
vertices not involved in $M$, because this would create an augmenting
path. So, in $G\left(N\right)$ the vertices of $T$ have degree less
than $n-1$, while in $\EGs$ there is no edge between vertices of
degree less than $n-1$. Therefore $G\left(N\right)\ncong\EGs$, as
desired.

\section{\label{sec:clique}Proof of \texorpdfstring{\cref{thm:clique}}{Theorem~\ref{thm:clique}}}

Let $f=n-2k$. In this section, we define a new hitting time: 
\[
\tau=\min\left\{ t:G\left(t\right)\text{ has exactly }f+1\text{ isolated vertices}\right\} .
\]
First we show that a.a.s.\ $\tau$ actually occurs.
\begin{lem}
\label{lem:tau-inf}
If $f=o\left(n\right)$ then a.a.s.\ $\tau<\infty$. Moreover, a.a.s.\ for $t\le \tau$, each $e(t)$ is accepted.
\end{lem}
\begin{proof}
Let
\[
\tau'=\min\left\{ t:G\left(t\right)\text{ has at most }f+2\text{ isolated vertices}\right\} .
\]
For $t\le\tau'$, each $G(t-1)$ has more than $f+2$ isolated vertices, so has no matching of size $k-1$, meaning that every edge $e\left(t\right)$ is accepted (so $\tau'<\infty$).
Adding an edge to a graph can destroy at most two isolated vertices,
so $G\left(\tau'\right)$ has $f+1$ or $f+2$ isolated vertices.
Now, condition on any outcome for $G\left(\tau'\right)$ with $f+2$
isolated vertices, and let $W$ be the set of these isolated vertices.
There are only $O\left(f^{2}\right)$ pairs of isolated vertices, but there are
$\left(f+2\right)\left(n-\left(f+2\right)\right)=\Omega \left(nf\right)=\omega \left(f^{2}\right)$
choices of an isolated and non-isolated vertex. Therefore we will
a.a.s.\ be offered an edge between $W$ and $V\setminus W$ before
we are ever offered an edge between isolated vertices. This edge (and all edges preceding it) will
be accepted, resulting in a graph with $f+1$ isolated vertices.
\end{proof}
We now compute the approximate value of $\tau$ so that we may compare
$G\left(\tau\right)$ to standard Erd\H os-R\'enyi/binomial random graphs. For $h>0$
define 
\[
t_{h}^{-}=\floor{\left(\log n-\log\left(f+1\right)-h\right)\frac{n}{2}},\quad t_{h}^{+}=\ceil{\left(\log n-\log\left(f+1\right)+h\right)\frac{n}{2}}.
\]

\begin{lem}
\label{lem:clique-tau-range}Suppose $f=o(n)$.
Then for any $h\to\infty$, a.a.s.\ $t_{h}^{-}\le\tau\le t_{h}^{+}$.
\end{lem}
\begin{proof}
We can assume $h$ is sufficiently slowly-growing so that $t_h^-=\omega(n)$. Now, let $X_{p}$ be the number of isolated vertices in $\GG\left(n,p\right)$. If $np\to\infty$ and $np=O(\log n)$ then
\begin{align*}
\E X_{p} & =n\left(1-p\right)^{n-1}\\
 & =n\left(e^{-p}+O\left(p^{2}\right)\right)^{n-1} =ne^{-pn}\left(1+O\left(p^{2}\right)\right)^{n}\\
 & =ne^{-pn}+O\left(n^{2}p^{2}e^{-pn}\right) =(1+o(1))ne^{-pn}.
\end{align*}
Note that $npe^{-np}=o(1)$, so a similar calculation gives
\[
\E\left[X_{p}\left(X_{p}-1\right)\right]  =n\left(n-1\right)\left(1-p\right)^{2n-3}
  \le n^2e^{-2pn}+O(n^3p^2e^{-2pn}) = (ne^{-pn})^2+o(ne^{-pn}).
\]
It follows that
\[
\Var X_{p}  =\E\left[X_{p}\left(X_{p}-1\right)\right]+\E X_{p}-\left(\E X_{p}\right)^{2} =o\left(n^2e^{-2pn}\right)=o\left(\E X_p^2\right).
\]
If $p\le\left(\log n-\log\left(f+1\right)-h\right)/n$ then $\E X_p=\omega(1)$ so a.a.s.\ $X_{p}\ge\E X_{p}/2\ge f+2$ by Chebyshev's inequality. If $p\ge\left(\log n-\log\left(f+1\right)+h\right)/n$ then
$\E X_p=O((f+1)e^{-h})=o(f+1)$ so a.a.s.\ $X_p\le f+1$ by Markov's inequality. Using say \cite[Proposition~1.15]{JLR}, in $G^{\all}(t_h^-)$ there are a.a.s.\ at least $f+2$ isolated vertices, and in $G^{\all}(t_h^+)$ there are a.a.s.\ at most $f+1$ isolated vertices. The first of these facts immediately implies $t_h^-<\tau$. Recalling from \cref{lem:tau-inf} that a.a.s.\ every edge up to time $\tau$ is accepted, the second of these facts implies that a.a.s.\ $t_h^+\ge \tau$.
\end{proof}
Now, given $G\left(\tau\right)$, we will have $G\left(N\right)\cong\EGc$
if and only if we reject all further edges involving an isolated vertex
of $G\left(\tau\right)$. The proofs of the positive and negative
parts of \cref{thm:clique} now diverge.

\subsection{Matching-resilience: the positive part}

In this section we explain how to prove that if $f=o\left(\sqrt{n}/\log n\right)$
then after time $\tau$ we a.a.s.\ reject all edges involving isolated
vertices.

Say a vertex in a graph is \emph{dangerous }if it has a neighbour
of degree 1, or if it has a neighbour of degree 2 and that neighbour
is within distance 2 of another vertex of degree at most 2. Say a
graph with an odd number of vertices is \emph{matching-resilient}
if after deleting any non-dangerous vertex there is a perfect matching.
\mk{is it appropriate to call this resilience, or does that word
only refer to deleting edges?} Let $W$ be the set of isolated vertices
in $G\left(\tau\right)$. The following lemma is crucial.
\begin{lem}
\label{lem:resilient-matching}If $f=o\left(\sqrt{n}/\log n\right)$,
then a.a.s.\ $G\left(\tau\right)\setminus W$ is matching-resilient.
\end{lem}
The proof of \cref{lem:resilient-matching} is a bit involved, so we
defer it to \cref{subsec:resilient-matching}. Next, we also need a
bound on the number of dangerous vertices. The slightly cumbersome
definition of a dangerous vertex was carefully chosen so that the
following lemma would hold.
\begin{lem}
\label{lem:dangerous-bound}If $f=o\left(\sqrt{n}/\log n\right)$,
then a.a.s.\ $G\left(\tau\right)$ has $o\left(\sqrt{n}\right)$ dangerous
vertices.
\end{lem}
\begin{proof}
First we introduce some convenient terminology. By ``1-vertex''
we mean a vertex with degree at most 1, by ``2-vertex'' we mean
a vertex with degree at most 2, and by ``2-pair'' we mean a pair
of 2-vertices whose distance is at most 2. Every dangerous vertex
is adjacent to a 1-vertex or adjacent to one of the vertices of a
2-pair, so it suffices to show that in $G\left(\tau\right)$ there
are $o\left(\sqrt{n}\right)$ 1-vertices and $o\left(\sqrt{n}\right)$
2-pairs.

Now, let $g=\sqrt{n}/\left(f\log n\right)\to\infty$ and choose $h\to\infty$
with $h=o\left(\log g\right)$. We consider the binomial random graph
$G\in\GG\left(n,p\right)$, with 
\[
p=\frac{\log n-\log\left(f+1\right)-h}{n}=\frac{\log n/2+\log\log n+\left(1+o\left(1\right)\right)\log g}{n}.
\]
The expected number of 1-vertices in $G$ is
\[
n\left(1-p\right)^{n-1}+n\left(n-1\right)p\left(1-p\right)^{n-2}=o\left(\sqrt{n}\right),
\]
and the expected number of 2-vertices is
\[
o\left(\sqrt{n}\right)+O\left(n^{3}p^{2}\left(1-p\right)^{n}\right)=o\left(\sqrt{n}\log n\right).
\]
Considering all possible cases for the structure of a 2-pair, \mk{would
a picture be helpful?} the expected number of 2-pairs is
\[
O\left(n^{2}\left(1-p\right)^{2n}\left(p+np^{2}+n^{2}p^{3}+n^{3}p^{4}\right)\right)=o\left(\log^{2}n\right).
\]
Let $m$ be
the number of edges of $G$; conditioned on $m$, $G$ has the Erd\H os-R\'enyi distribution $\GG(n,m)$, so we can couple $G$ and $\left(G^{\all}\left(t\right)\right)_{t}$
in such a way that $G^{\all}\left(m\right)=G$. By the Chernoff bound,
a.a.s.
\begin{align*}
\left|m-\left(\log n-\log\left(f+1\right)-h\right)\frac{n}{2}\right| & =o\left(n\right),
\end{align*}
meaning that $m<t^-_{h/2}$. So condition on an outcome of $G=G^{\all}\left(m\right)$ satisfying
this property, such that $G$ has at most $\sqrt{n}$ 1-vertices,
at most $\sqrt{n}\log n$ 2-vertices, and at most $\log^{2}n$ 2-pairs.
Let $U\left(t\right)$ be the set of 2-vertices and their neighbours
in $G^{\all}\left(t\right)$. Note that $G^{\all}\left(t\right)$
has at most $\sqrt{n}$ 1-vertices, and $\left|U\left(t\right)\right|\le3\sqrt{n}\log n$,
for all $t\ge m$. Now, after time $m$, the only way a new 2-pair
can be formed is if we are offered an edge $e\left(t\right)$ between
two vertices in $U\left(t-1\right)$. Note that $t_{h/2}^{+}-m=O\left(hn\right)$,
so the expected number of such edges we will be offered before time
$t_{h/2}^{+}$ is
\[
O\left(\frac{\left(\sqrt{n}\log n\right)^{2}hn}{n^{2}}\right)=O\left(h\log^{2}n\right)=o\left(\log^{3}n\right).
\]
So, a.a.s.\ there are at most $\log^{3}n$ times $m<t\le t_{h/2}^{-}$
in which new 2-pairs are created. Now, we will soon see in \cref{lem:properties-A}
that a.a.s.\ in each of $G^{\all}\left(t_{h/2}^-\right),G^{\all}\left(t_{h/2}^-+1\right),\dots,G^{\all}\left(t_{h/2}^{+}\right)$
there are no three 2-vertices within distance 20 of each other, which
means that only one new 2-pair can be created at a time. It follows
that a.a.s.\ each of $G^{\all}\left(t_{h/2}^-\right),G^{\all}\left(t_{h/2}^-+1\right),\dots,G^{\all}\left(t_{h/2}^{+}\right)$
have $o\left(\sqrt{n}\right)$ 1-vertices and $O\left(\log^{3}n\right)$
2-pairs. By \cref{lem:clique-tau-range}, a.a.s.\ $t_{h/2}^-\le\tau\le t_{h/2}^{+}$,
implying that a.a.s.\ $G\left(\tau\right)$ has $o\left(\sqrt{n}\right)$
dangerous vertices, as desired.
\end{proof}
In view of the above lemmas, condition on an outcome of $\tau,e\left(1\right),\dots,e\left(\tau\right)$
such that $\tau\le n\log n$, and $G\left(\tau\right)$ is matching-resilient,
and $G\left(\tau\right)$ has $o\left(\sqrt{n}\right)$ dangerous
vertices. Let $W$ and $U$ be the sets of isolated and dangerous
vertices of $G\left(\tau\right)$, respectively, let $V=[n]$ be the set of all vertices, and let $r=2n\log n$.
The probability we are offered an edge between $W$ and $U$ before
time $\tau+r$ is
\[
o\left(\frac{r(f+1)\sqrt{n}}{n^{2}}\right)=o\left(1\right).
\]
It follows that a.a.s.\ $G\left(\tau+r\right)$
still has $f+1$ isolated vertices. Indeed, any edge $e(t)$ between $W$ and $V\setminus U$ is rejected, beacause by matching-resilience $G(\tau)$ has a $(k-1)$-matching not involving the vertices of $e(t)$. Next, the expected number of edges
we are offered before time $\tau+r$ that involve $W$ at all is $O\left(rf/n\right)=o\left(r\right)$.
So, by Markov's inequality, a.a.s.\ we are offered fewer than $r/2$
such edges, meaning that we are offered at least $n\log n$ edges within $V\setminus W$.
As long as each vertex in $W$ remains isolated, every edge within
$V\setminus W$ will be accepted.

Let $G$ be the Erd\H os-R\'enyi random graph on the vertex set
$V\setminus W$ with $n\log n$ random edges. By the above considerations,
$G\left(\tau+r\right)\setminus W$ can be coupled with $G$ in such
a way that a.a.s.\ $G\left(\tau+r\right)\setminus W\supseteq G$.
It is well known that $G$ a.a.s.\ has a Hamilton cycle (see for example
\cite[Section~5.1]{JLR}), which means that it has a perfect matching
after deleting any vertex. The same is a.a.s.\ true for $G\left(\tau+r\right)\setminus W$,
meaning that no edge involving an isolated vertex can ever be accepted after this point.

\subsection{Cherries: the negative part}

In this subsection we will prove that if $f=\omega\left(\sqrt{n}/\log n\right)$
and $f=o(n)$ then a.a.s.\ we accept some
edge involving an isolated vertex, after time $\tau$. This suffices, because the case where $f=\Omega(n)$ is handled by \cref{thm:not-extremal}.

We say a path of length 2 in a
graph is a \emph{cherry} if its two endpoints have degree 1 in the
graph. A matching can use at most two of the three vertices of a cherry,
so if $G\left(\tau\right)$ has a cherry then there is some freedom
to add an edge involving an isolated vertex, without creating a $k$-matching.
\begin{lem}
\label{lem:cherries-persist}Suppose $f=\omega\left(\sqrt{n}/\log n\right)$.
Let $g=f\log n/\sqrt{n}\to\infty$, and choose $h\to\infty$ to satisfy
$h=o\left(\log g\right)$. Then a.a.s.\ $G\left(t_{h}^{+}\right)$ has a cherry.
\end{lem}
\begin{proof}
We use a two-phase argument similar to the proof of \cref{lem:isolated-triangle}.
Consider $G_1\in\GG\left(n,p_1\right)$ and $G_2\in\GG\left(n,p_2\right)$,
with 
\[
p_1=\frac{\log n-\log\left(f+1\right)-h}{n}=\frac{\log n/2+\log\log n-\left(1+o\left(1\right)\right)\log g}{n},
\]
 and $p_2\le 5h/n$ chosen such that $p_2(1-p_1)=4h/n$. Then by the Chernoff bound, a.a.s.
\begin{align*}
\left|e\left(G_1\right)-\left(\log n-\log\left(f+1\right)-h\right)\frac{n}{2}\right| & =o\left(n\right),\\
\left|e\left(G_2\setminus G_1\right)-2hn\right| & =o\left(n\right),
\end{align*}
implying that $e(G_1)<t_{h}^{-}$ and $e(G_1\cup G_2)>t_{h}^{+}$. Conditioned on $e(G_1)$ (respectively $e(G_1\cup G_2)$), note that $G_1$ (respectively $G_1\cup G_2$) has an Erd\H os-R\'enyi random graph distribution, so we can couple $G_1$, $G_2$ and $\left(G\left(t\right)\right)_{t}$
such that a.a.s.\ $G_1\subseteq G^\all\left(t_{h}^{-}\right)=G\left(t_{h}^{-}\right)$ and $G_1\cup G_2\supseteq G^\all\left(t_{h}^{+}\right)\supseteq G\left(t_{h}^{+}\right)$.
We will prove that a.a.s.\ $G_1$ has many cherries, and a.a.s.\ at least one of these
cherries remains in $G_2\cup G_1$.

Let $X$ be the number of cherries in $G_1$. Then
\begin{align*}
\E X & =n{n-1 \choose 2}p_1^{2}\left(1-p_1\right)^{2n-5}=\Omega\left(n^2 p_1^2 g^2/\log^2 n\right).
\end{align*}
Considering separately the cases $g\ge \log^2 n$ and $g< \log^2 n$, noting that $np_1\to \infty$, it follows that $\E X=\omega(g)$. Considering all the possible ways a pair of distinct cherries can intersect, we can compute
\begin{align*}
\E[X(X-1)]&=\left(n\binom{n-1}2 (n-3)\binom{n-4}2\right)p_1^4(1-p_1)^{4n-14}+O(n^5p_1^4(1-p_1)^{4n}+n^4p_1^3(1-p_1)^{3n})\\
&=(1-o(1))(\E X)^2+O(np_1e^{-np_1}\E X)=(1-o(1))(\E X)^2,\\
\Var X&=o\left(\left(\E X\right)^{2}\right).
\end{align*}
Therefore a.a.s.\ $G_1$
has at least $\E X/2>2g$ cherries. We say a pair of cherries is ``externally intersecting''
if they intersect in their degree-1 vertices (the only way this can occur is if the union of the two is a 3-edge star). Let $Z$ be the number
of pairs of externally intersecting cherries; then we can compute
\begin{align*}
\E Z & =O\left(n^{4}p_1^{3}\left(1-p_1\right)^{3n}\right)=O(np_1e^{-np_1}\E X)=o\left(\E X\right).
\end{align*}
So, a.a.s.\ there is a collection of $g$ cherries in $G_1$ which are pairwise externally disjoint.
Condition on such an outcome of $G_1$. Let $Y$ be the number of these cherries in this collection that remain in
$G_1\cup G_2$. Then, we can compute
\begin{align*}
\E Y&=g\left(1-p_2\right)^{2n-5}=\Omega\left(ge^{-5h}\right)=g^{1-o\left(1\right)}=\omega(1),\\
\E[Y(Y-1)]&=g(g-1)(1-p_2)^{4n-18}=(1-o(1))(\E Y)^2,\\
\Var(Y)&=o\left((\E X)^2\right).
\end{align*}
So, a.a.s.\ $Y>0$.
\end{proof}
Now, let $P_{G(\tau)}$ be the set of pairs consisting of an isolated vertex of $G(\tau)$ and a vertex of degree at least 2 in $G(\tau)$. Let $h$ be as in \cref{lem:cherries-persist}. We claim that in order to prove that a.a.s.\ some edge involving an isolated vertex of $G(\tau)$ is accepted, it suffices to show that a.a.s.\ there is a time $t<t_h^+$ for which an element of $P_{G(\tau)}$ is offered as $e(t)$. Indeed, let $e(t)$ be such an element, and suppose that $G(t_h^+)$ has $f+1$ isolated vertices (otherwise we are done). By \cref{lem:cherries-persist}, a.a.s. $G(t_h^+)+e(t)$ has a cherry, and at most two of the three vertices of this cherry can be used in a matching. Since $G(t_h^+)+e(t)$ has $f=n-2k$ isolated vertices, it is $k$-matching-free, so $e(t)$ is accepted at time $t$.

In view of the above discussion, condition on $\tau\le t^+_{h/2}$ such that no edge of $P_{G(\tau)}$ has been offered yet. The probability we are offered an edge of $P_{G(\tau)}$ before time $t^+_h$ is at least
\[
1-\left(1-\frac{\left|P_{G(\tau)}\right|}{\binom{n}{2}}\right)^{hn/2}\ge 1-e^{-\Omega\left(\left|P_{G(\tau)}\right|h/n\right)},
\]
so it actually suffices to show that a.a.s.\ $\left|P_{G(\tau)}\right|=\Omega(n)$. But this is an immediate consequence of the fact that a.a.s.\ $G\left(\tau\right)$ has
$\Omega\left(n\right)$ vertices of degree at least 2. This can be
proved by applying the second moment method to the number of isolated vertices and to the number of degree-1 vertices in $G\left(t_{h}^{-}\right)$, for sufficiently slowly-growing $h$ such that $t_h^-=\omega(n)$.

\subsection{Proof of \texorpdfstring{\cref{lem:resilient-matching}}{Lemma~\ref{lem:resilient-matching}}\label{subsec:resilient-matching}}

Let $g=\sqrt{n}/\left(f\log n\right)\to\infty$ and choose $h\to\infty$
to satisfy $h=o\left(\log g\right)$. Let $W\left(t\right)$
be the set of isolated vertices in $G^{\all}\left(t\right)$. We will
prove that a.a.s.\ for each $t_{h}^{-}\le t\le t_{h}^{+}$, if $n-\left|W\left(t\right)\right|$
is odd then $G^{\all}\left(t\right)\setminus W\left(t\right)$ is
matching-resilient.

To accomplish this, we adapt the method of \L uczak and Ruci\'nski
\cite{LR} used to study tree-packings. For the special case of matchings,
this method was outlined in \cite[Section~4.1]{JLR}. Where possible,
we will re-use lemmas from \cite{LR} and \cite{JLR}.

\mk{Perhaps I should explain \cite{LR} more? Eg. in the case of
matchings (tree-packing where the tree is an edge), say what the values
of certain parameters are, so that the reader can more easily interpret
the lemmas in \cite{LR}?}

First, the following lemma follows directly from parts (i) and (ii)
of \cite[Theorem~3]{LR}.
\begin{lem}
\label{lem:no-cherries}For any $h\to\infty$, a.a.s.\ for each $t\ge\left(\log n/2+\log\log n+h\right)n/2$,
the largest component of $G^{\all}\left(t\right)$ has no cherries.
\end{lem}
The next lemma follows directly from \cite[Lemma~1]{LR}.
\begin{lem}
\label{lem:only-isolated}For any $h\to\infty$, a.a.s.\ for each $t\ge\left(\log n/2+\log\log n/2+h\right)n/2$,
in $G^{\all}\left(t\right)$ there are only isolated vertices outside
the largest component.
\end{lem}
The next lemma is a slight adaptation of \cite[Lemma~4.7]{JLR} (essentially the only difference is that we need a certain property to hold for a range of $G^{\all}(t)$ instead of a single random graph).
\begin{lem}
\label{lem:PM-min-degree}For any $c>0$, a.a.s.\ for all $n\log n/4\le t\le n\log n$,
every bipartite subgraph with minimum degree at least $c\log n$, induced in $G^{\all}\left(t\right)$ by two sets of equal size,
contains a perfect matching.
\end{lem}
\begin{proof}
From \cite[Lemma~4]{LR}, we know that a.a.s.\ for all such $t$,
\begin{enumerate}
\item [(1)]for every pair of disjoint subsets of size $n\left(\log\log n\right)^{2}/\log n$,
there is an edge between them in $G^{\all}\left(t\right)$,
\item [(2)]every set $S$ of at most $2n\left(\log\log n\right)^{2}/\log n$
vertices induces fewer than $\left(\log\log n\right)^{3}\left|S\right|$
edges in $G^{\all}\left(t\right)$.
\end{enumerate}
We can then conclude the proof exactly as in \cite[Lemma~4.7]{JLR}, using Hall's theorem.
\end{proof}
The next lemma follows from \cite[Lemma~4]{LR}. Let $A=\left\{ 1,\dots,\floor{n/2}\right\} $
and $B=\left\{ \floor{n/2}+1,\dots,n\right\} $, so that $A$ and
$B$ give a fixed balanced partition of $\range n$. As in \cite[Section~3]{LR},
say a vertex is \emph{bad} if it has fewer than $\log n/200$ neighbours
in $A$ or in $B$. Say a vertex is \emph{small} if $d\left(v\right)<40$.
\begin{lem}
\label{lem:properties-A}A.a.s.\ for all $n\log n/4\le t\le n\log n$,
the following properties hold.
\begin{enumerate}
\item [(1)]$G^{\all}\left(t\right)$ has no more than $n/\log^{40}n$ bad
vertices,
\item [(2)]$G^{\all}\left(t\right)$ has no $8$ bad vertices within distance
$20$ from each other,
\item [(3)]$G^{\all}\left(t\right)$ has no $2$ small and $1$ bad vertices
within distance $20$ from each other,
\item [(4)]$G^{\all}\left(t\right)$ has maximum degree less than $6\log n$.
\end{enumerate}
\end{lem}
Now we can prove \cref{lem:resilient-matching}. We closely follow
the proof of \cite[Theorem~4.4]{JLR}.
\begin{proof}[Proof of \cref{lem:resilient-matching}]
A.a.s.\ For any $t_{h}^{-}\le t\le t_{h}^{+}$ the graph $G^{\all}\left(t\right)$
satisfies the properties of \cref{lem:no-cherries,lem:only-isolated,lem:PM-min-degree,lem:properties-A},
so we assume these properties hold for the remainder of the proof.
Consider any $t_{h}^{-}\le t\le t_{h}^{+}$, suppose $n-\left|W\left(t\right)\right|$
is odd, and let $G=G^{\all}\left(t\right)\setminus W\left(t\right)$.
Let $v$ be any non-dangerous vertex. We will show that $G\setminus v$
has a perfect matching.

Order the bad vertices of $G$ by degrees:
\[
d\left(v_{1}\right)\le d\left(v_{2}\right)\le\dots\le d\left(v_{\ell}\right).
\]
We greedily match these vertices one-by-one with vertices $u_{1},\dots,u_{\ell}$,
as follows. Suppose $v_{1},\dots,v_{i-1}$ are already matched with
$u_{1},\dots,u_{i-1}$ (some $v_{j}$ may be matched with some $v_{q}$,
which means $u_{j}=v_{q}$ and $u_{q}=v_{j}$). Let $V_{i-1}=\left\{ v_{1},\dots,v_{i-1}\right\} $
and $U_{i-1}=\left\{ u_{1},\dots,u_{i-1}\right\} $. Then, choose
$u_{i}$ to be an (arbitrary) vertex of smallest degree among $N_{G}\left(v_{i}\right)\setminus\left(V_{i-1}\cup U_{i-1}\cup\left\{ v\right\} \right)$.
We need to show that this set is always nonempty, so that this choice
is always possible.
\begin{itemize}
\item If $d\left(v_{i}\right)=1$ then $u_{i}$ is the unique neighbour
of $v_{i}$. By definition $u_{i}$ is dangerous, so $u_{i}\ne v$.
Since $G$ has no component of size 2, $u_{i}$ does not have degree
1. Since the $v_{i}$ are ordered by degrees, this means $u_{i}\notin V_{i-1}$.
Since $G$ has no cherry, $u_{i}\notin U_{i-1}$.
\item Suppose $d\left(v_{i}\right)=2$. Since there are no three small vertices
within distance 4 of each other, there is at most one element of $V_{i-1}$
within distance 2 of $v_{i}$. We consider three cases.
\begin{itemize}
\item If $v_{i}$ has no neighbour in $V_{i-1}\cup U_{i-1}$, then $v_{i}$
has at least one neighbour other than $v$, which is a viable choice
for $u_{i}$.
\item If $v_{i}$ has exactly one neighbour in $V_{i-1}\cup U_{i-1}$, then
some vertex in $V_{i-1}$ is within distance 2 of $v_{i}$. By definition
the other neighbour of $v_{i}$ is dangerous, so it cannot have been
chosen for $v$ and is a viable choice for $u_{i}$.
\item The remaining case is that $v_{i}$ has two neighbours in $V_{i-1}\cup U_{i-1}$. These neighbours cannot be of the form $\{v_j,v_q\}$ or $\{v_j,u_q\}$ or $\{u_j,u_q\}$ for $j\ne q$, as this would imply two different elements of $V_{i-1}$
within distance 2 of $v_{i}$. So, it actually remains to consider the case $N_{G}\left(v_{i}\right)=\left\{ v_{j},u_{j}\right\} $
for some $j<i$. This would mean that $v_{i}$ was a viable choice
for $u_{j}$ but was not chosen, which means $d\left(u_{j}\right)\le2$.
But this would give three small vertices at distance 1 of each
other, which is impossible.
\end{itemize}
\item If $3\le d\left(v_{i}\right)\le40$ then by the same reasoning as
above, $v_{i}$ has at most one neighbour in $V_{i-1}\cup U_{i-1}$.
It's possible that $v$ is also a neighbour of $v_{i}$, but there
is still at least one neighbour left for $u_{i}$.
\item If $d\left(v_{i}\right)\ge41$, then $v_{i}$ has at most 13 neighbours
in $V_{i-1}\cup U_{i-1}$, since otherwise there would be 7 bad vertices
within distance 2 of $v_{i}$ (and therefore 8 bad vertices within
distance 4 of each other). So there are plenty of neighbours left
for $u_{i}$.
\end{itemize}
We have proved that all bad vertices can be matched. After removing all
vertices matched so far (and $v$), every vertex has at least $\log n/200-17$ neighbours
in $A$ and in $B$ (no neighbourhood loses more than 17 vertices, or else there would be more than 8
bad vertices within distance 4 of each other). The remaining vertices
in $A$ and $B$ may no longer form a balanced bipartition. In order
to apply the property in \cref{lem:PM-min-degree}, we move some carefully
chosen vertices across the partition to balance it. A \emph{2-independent
set} of vertices is a set of vertices such that no two share a common
neighbour. Recalling that $G$ has maximum degree at most $6\log n$,
it has a 2-independent set of size $\Omega\left(n/\log^{2}n\right)$,
which is more than the $O\left(n/\log^{40}n\right)$ vertices we must
move to balance the bipartition. So, move a 2-independent set of vertices
to balance the partition. At most one neighbour of each vertex is moved, so the edges between the parts form a balanced
bipartite graph with minimum degree at least $\log n/200-18$. We can then apply the property in \cref{lem:PM-min-degree} to see that this graph has a perfect matching, finishing the proof.
\end{proof}

\section{\label{sec:eps}Proof of \texorpdfstring{\cref{thm:eps}}{Proposition~\ref{thm:eps}}\label{sec:eps-n}}

In this section we prove \cref{thm:eps}. Note that we may assume $k=\omega(\sqrt n)$ (otherwise we can defer to \cref{thm:o(n)}), and we may assume that $k/n\le 1/100$ (because if $k/n=\Omega(1)$ then \cref{thm:eps} is trivial). Also, note that it actually suffices
to prove that a.a.s.\ $G\left(N\right)$ has $k-f$ vertices of degree
$n-1$, for some $f=O(k^2/n)$. Indeed, since $G\left(N\right)$ is $k$-matching-free, within
the remaining $n-k+f$ vertices there is no matching of size $f$,
meaning that there must be an independent set of size $n-k+f-2\left(f-1\right)>n-k-f$.

Define $\tau$ as in \cref{subsec:o(n)-initial}. We first need a stronger version of \cref{lem:o(n)-tau-bound}.

\begin{lem}
\label{lem:eps-tau-bound}A.a.s.\ $\tau\le(1+O(k/n))k$.
\end{lem}
\begin{proof}
Let $q=(1+11k/n)k$. We proceed in basically the same way as \cref{lem:o(n)-tau-bound}. For each $t\le q$, $G^\all(t-1)$ has at most $q$ edges, comprising at most $2k$ vertices. The probability that $e(t)$ does not increase the matching number is at most the probability that it intersects the edges of $G^\all(t-1)$, which is at most $2qn/\left(\binom{n}{2}-q\right)\le 9k/n$. The number of such steps is stochastically dominated by the binomial distribution $\Bin(q,9k/n)$, which has expectation $9qk/n\le 10k^2/n$. So, recalling the assumption that $k^2/n=\omega(1)$, by the Chernoff bound a.a.s.\ there are at most $11k^2/n$ steps among the first $q$ that do not increase the matching number. It follows that $\tau\le q$.
\end{proof}

Now, condition on an outcome of $\tau,e(1),\dots,e(\tau)$ such that $\tau\le(1+O(k/n))k$, and recall the definitions of $M,A,B$ and augmenting paths from \cref{subsec:o(n)-initial}. In $G^\all(\tau)$, note that there are only $O(k^2/n)$ edges not in $M$, and each such edge can be incident to at most two edges of $M$. Therefore, $(1-O(k/n))k$ of the $k-1$ edges of $M$ are isolated in $G^\all(\tau)$. Let $M'$ be the sub-matching of such edges.
For each $e\in M'$, let $E_e$ be the event that we are offered an edge between  some $b\in B$ and some endpoint $a^+_e$ of $e$ (let $a^-_e$ be the other endpoint of $e$) before we are ever offered any edges between $e$ and $A$, and then we are offered another edge between $a^+_e$ and $B$ before we are offered the edge $\{a^-_e,b\}$ or any further edges between $e$ and $A$. There are $2|B|$ possible edges between $e$ and $B$, out of $2(n-2)$ possible edges involving $e$ other than $e$ itself. After $\{a^+_e,b\}$ has been revealed, there are $|B|-1$ possible further edges between $a^+_e$ and $B$, out of a total of $n-3$ possible further edges involving $a^+_e$. There are also $|A|-2$ edges between $a^-_e$ and $A$. So,
\begin{equation}
\Pr(E_e)=\frac{2|B|}{2(n-2)}\cdot\frac{|B|-1}{(n-3)+1+(|A|-2)}=1-O\left(\frac kn\right).\label{eq:Ee}
\end{equation}
Now, consider distinct $e,e'\in M'$. We will show that $E_e$ and $E_{e'}$ are essentially independent. For a possible edge $f\subseteq A$, let $Q^e_f$ be the event that $f$ is not offered until three edges have already been offered between $e$ and $B$ (this means that $f$ is not offered until $E_e$ has already been determined). Recalling $|B|=\Theta(n)$, we have $$\Pr(Q_f^e)=\frac{2|B|}{2|B|+1}\cdot\frac{2|B|-1}{2|B|}\cdot\frac{2|B|-2}{2|B|-1}=1-O\left(\frac1n\right).$$
Now, let $Q$ be the intersection of the events $Q_f^e\cap Q_f^{e'}$ for each of the four possible edges $f$ between $e$ and $e'$. If $Q$ holds, then none of these four edges are offered until $E_e$ and $E_{e'}$ have already been determined. By the union bound, $\Pr(Q)=1-O(1/n)$.  Let $E_e'$ and $E_{e'}'$ have the same definitions as the events $E_e$ and $E_{e'}$, but ignoring all edges between $e$ and $e'$.  This means that $E_e'$ and $E_{e'}'$ are independent. Now, note that
\[
\Pr(E_e\cap E_{e'})=O(\Pr(\overline{Q}))+\Pr(E_e\cap E_{e'}\cap Q)=O\left(\frac1n\right)+\Pr(E_e\cap E_{e'}\cap Q),
\]
and similarly $\Pr(E_e'\cap E_{e'}')=O\left(1/n\right)+\Pr(E_e'\cap E_{e'}'\cap Q)$.  Observe that $E_e\cap E_{e'}\cap Q$ is actually the same event as $E_e'\cap E_{e'}'\cap Q$, so
$$\Pr(E_e\cap E_{e'})=\Pr(E_e'\cap E_{e'}')+O\left(\frac1n\right)=\Pr(E_e')\Pr(E_{e'}')+O\left(\frac1n\right)=\Pr(E_e)^2+O\left(\frac1n\right).$$
Let $X$ be the number of edges $e\in M'$ such that $E_e$ holds. By the above calculations, $\E X=k-O(k^2/n)$ and
$$\Var X=\!\!\sum_{(e,e')\in E(M')^2}\left(\Pr(E_e\cap E_{e'})-\Pr(E_e)\Pr(E_{e'})\right)=O(k^2/n).$$
 Recalling the assumption that $k^2/n=\omega(1)$, a.a.s.\ $X\ge \E X-k^2/n$ by Chebyshev's inequality, implying that $E_e$ holds for $(1-O(k/n))k$ edges $e$ of $M'$. Now, the following lemma completes the proof.

\begin{lem}
If $E_e$ holds, then $a^+_e$ has degree $n-1$ in $G(N)$.
\end{lem}
\begin{proof}
If $E_e$ holds, the first edge between $a^+_e$ and $B$ will be accepted, all subsequent edges between the other endpoint $a^-_e$ of $e$ and $B$ will be rejected (because they would create a length-3 augmenting path), and the next edge between $a^+_e$ and $B$ will be accepted (at this point the connected component of $e$ will then be a 3-edge star, involving two vertices $b_1,b_2\in B$). Now, suppose for the purpose of contradiction that some further edge $e(t)$ involving $a^+_e$ is rejected. This means that $e(t)$ would introduce an augmenting path $P$ starting at some vertex in $B$, passing through $e$ and $e(t)$ consecutively, then ending at some vertex $b^*$ in $B$. Without loss of generality suppose $b_1\ne b^*$; but then $G(t-1)$ already had an augmenting path, obtained by replacing the portion of $P$ after $e$ by the edge $\{a^+_e,b_1\}$. This is a contradiction.
\end{proof}

\section{\label{sec:not-extremal}Proof of \texorpdfstring{\cref{thm:not-extremal}}{Proposition~\ref{thm:not-extremal}}}

First, claim (1) will be an immediate consequence
of two simple lemmas. 
\begin{lem}
A.a.s.\ $G\in\GG\left(n,6n\right)$
has no independent set of size $n/3$.
\end{lem}
\begin{proof}
Let $p=6n/\binom{n}{2}$; using say \cite[Corollary~1.16]{JLR} it
suffices to show that $G'\in\GG\left(n,p\right)$ has the required
property. Let $X$ be the number of independent sets of size $n/3$ in $G'$; then using Stirling's approximation we have
$$\E X=\binom{n}{n/3}(1-p)^{\binom{n/3}{2}}\le \exp \left(-\left(\frac 1 3 \log{\frac 1 3}+\frac 2 3 \log{\frac 2 3}+\frac 23+o(1)\right)n\right)=o(1),$$ so the desired result follows from Markov's inequality.
\end{proof}
\begin{lem}
A.a.s.\ the largest matching in $G\in\GG\left(n,6n\right)$
has size at most $n/2-e^{-13}n/2$.
\end{lem}
\begin{proof}
Let $p=6n/\binom{n}{2}$; it suffices to show that $G'\in\GG\left(n,p\right)$
has the required property. Let $X$ be the number of isolated vertices
in $G'$; then we compute $\E X=n\left(1-p\right)^{n-1}=\left(1-o\left(1\right)\right)e^{-12}n=\Omega\left(n\right)$,
whereas one can check that $\Var X=o\left(n^{2}\right)=o\left(\left(\E X\right)^{2}\right)$. So,
a.a.s.\ $G'$ has $e^{-13}n$ isolated vertices, which
cannot contribute to a matching.
\end{proof}
Recall the definition of $\tau$ from \cref{sec:o(n),sec:eps}; we have proved
that for $k\ge n/2-e^{-13}n/2$, a.a.s.\ $G^{\all}\left(\tau\right)=G\left(\tau\right)\subseteq G\left(N\right)$
has no independent set of size $n/3=n-k-\Omega\left(n\right)$, proving (1).

Now we consider claim (2). Say a 2-path in a graph is \emph{isolated}
if there are no edges between the 2-path and the rest of the graph.
\begin{lem}
Consider any constants $0<R_{1}<R_{2}$. Then a.a.s.\ each $G^{\all}\left(t\right)$,
for $R_{1}n\le t\le R_{2}n$, has $\Omega\left(n\right)$ isolated
2-paths.
\end{lem}
\begin{proof}
Consider $R_{1}n\le t\le R_{2}n$. To use the union bound, it suffices to show that $G^{\all}(t)$ has the required property with probability $1-o\left(1/n\right)$. Let $p=t/{n \choose 2}$; by Pittel's inequality (see \cite[Section~1.4]{JLR}), it actually suffices to show this for $G'\in\GG\left(n,p\right)$, instead. Let $X$
be the number of isolated 2-paths in $G'$, so that
\[
\E X=3{n \choose 3}p^{2}\left(1-p\right)^{3\left(n-3\right)}=\Omega\left(n\right).
\]
Observe that changing the status of any edge changes the value of
$X$ by at most 2, so by a Bernstein-type concentration inequality (see for example \cite[Theorem~2.11]{Kwa}),
\[
\Pr\left(X\le\E X/2\right)\le \exp\left(-\frac{(\E X/2)^2}{16t+2\E X}\right)=e^{-\Omega\left(n\right)},
\]
as desired.
\end{proof}
Let $R_{1}=(k-1)/n=\Theta\left(1\right)$, and again recall the definition of $\tau$ from \cref{sec:o(n),sec:eps}. Using \cref{lem:n-tau-bound}, we know
that there is $R_{2}$ such that a.a.s.\ $\tau\le R_{2}n$, and trivially (as remarked in \cref{sec:o(n)}) we have $\tau \ge k-1=R_{1}n$. It follows that that a.a.s.\ $G^{\all}\left(\tau\right)=G\left(\tau\right)$
has $\Omega\left(n\right)$ isolated 2-paths. For any maximum $\left(k-1\right)$-matching
in $G\left(\tau\right)$, there is at least one vertex in every isolated
2-path which does not contribute to that matching, so there are a.a.s.\ $2\left(k-1\right)+\Omega\left(n\right)$ non-isolated vertices. This proves claim (2).

\section{Concluding remarks}

In this paper we studied the random greedy $k$-matching-free process,
in which edges are iteratively added to an empty graph, each chosen
uniformly at random subject to the restriction that no $k$-matching
is formed. We discovered that if $k=o\left(n\right)$ or $n-2k=o\left(\sqrt{n}/\log n\right)$
then this process is likely to produce an extremal $k$-matching-free
graph. We also made a first step exploring the intermediate regime, but here there is much more work to be done. In particular, \cref{thm:not-extremal} says that there is a range of values of $k$ for which the outcome of the $k$-matching-free process is likely to be far from an
extremal graph; we wonder whether these random graphs have interesting properties that may be useful for other problems.

We also hope that the ideas in this paper may be useful for studying
other related kinds of random processes. For example, it is natural
to ask about the $k$-path-free process, or the restricted-girth process
where we greedily add edges keeping the girth above some value $k$.
(The restricted-girth process has already been studied for fixed $k$
by Osthus and Taraz \cite{OT}; see also the work of Bayati, Montanari
and Saberi \cite{BMS} on a slightly different process.) There are
also natural generalisations of these processes to hypergraphs. In
particular, define the \emph{$\left(-2\right)$-girth }of a 3-uniform
hypergraph to be the smallest integer $g\ge4$ such that there is
a set of $g$ vertices spanning at least $g-2$ edges. Erd\H os \cite{Erd73}
asked in 1973 whether there are hypergraphs with large girth and quadratically
many edges; in an earlier version of this paper we suggested that analysis of a hypergraph generalisation of the restricted-girth
process might lead to progress on this question. Since that time, Glock, K\"uhn, Lo and Osthus \cite{GKLO} managed to prove Erd\H os' conjecture in precisely this way. (The authors also mention that the same result was independently proved by Bohman and Warnke).

Finally, we remarked in the introduction that despite the matching number and vertex cover number of a graph being very closely related to each other, the restricted covering process exhibits quite trivial behaviour compared to the matching-free process. Perhaps it would be interesting to explore more closely the relationship between these two parameters by considering random graph models where the vertex cover number and matching number are constrained to be equal (for example, one could consider a random process where edges are added as long as they do not separate the vertex cover number and matching number).


\begin{thebibliography}{10}

\bibitem{BMS}
M.~Bayati, A.~Montanari, and A.~Saberi, Generating random graphs with large
  girth, Proceedings of the Twentieth Annual ACM-SIAM Symposium on Discrete
  Algorithms, SIAM, 2009, pp.~566--575.

\bibitem{Berge}
C.~Berge, Two theorems in graph theory, \emph{Proceedings of the National
  Academy of Sciences} \textbf{43} (1957), no.~9, 842--844.

\bibitem{Boh}
T.~Bohman, The triangle-free process, \emph{Advances in Mathematics}
  \textbf{221} (2009), no.~5, 1653--1677.

\bibitem{BFKLS}
T.~Bohman, A.~Frieze, M.~Krivelevich, P.-S. Loh, and B.~Sudakov, Ramsey games
  with giants, \emph{Random Structures \& Algorithms} \textbf{38} (2011),
  no.~1-2, 1--32.

\bibitem{BFL}
T.~{B}ohman, A.~Frieze, and E.~Lubetzky, Random triangle removal,
  \emph{Advances in Mathematics} \textbf{280} (2015), 379--438.

\bibitem{BFMRS2}
T.~Bohman, A.~Frieze, R.~Martin, M.~Ruszink{\'o}, and C.~Smyth, Randomly
  generated intersecting hypergraphs {II}, \emph{Random Structures \&
  Algorithms} \textbf{30} (2007), no.~1-2, 17--34.

\bibitem{BK}
T.~Bohman and P.~Keevash, The early evolution of the ${H}$-free process,
  \emph{Inventiones Mathematicae} \textbf{181} (2010), no.~2, 291--336.

\bibitem{BK2}
\bysame, Dynamic concentration of the triangle-free process, \emph{arXiv
  preprint arXiv:1302.5963} (2013).

\bibitem{BMP}
T.~Bohman, D.~Mubayi, and M.~Picollelli, The independent neighborhoods process,
  \emph{Israel Journal of Mathematics} \textbf{214} (2016), no.~1, 333--357.


\bibitem{BolChapter}
B.~Bollob{\'a}s, Extremal graph theory, \textbf{Handbook of combinatorics}
  (R.~Graham, M.~Gr\"otschel, and L.~Lov\'asz, eds.), vol.~2, Elsevier, 1995,
  pp.~1231--1292.

\bibitem{BolEGT}
\bysame, \textbf{Extremal graph theory}, Academic Press, 1978.

\bibitem{BR}
B.~Bollob{\'a}s and O.~Riordan, Constrained graph processes, \emph{Electronic
  Journal of Combinatorics} \textbf{7} (2001), no.~1, \#R18.

\bibitem{Erd73}
P.~Erd\H{o}s, Problems and results in combinatorial analysis, \textbf{Colloquio
  internazionale sulle teorie combinatorie ({R}ome 1973), vol {II}}, Atti dei
  Convegni Lincei, vol.~17, 1973, pp.~3--17.

\bibitem{EG}
P.~Erd{\H{o}}s and T.~Gallai, On maximal paths and circuits of graphs,
  \emph{Acta Mathematica Academiae Scientiarum Hungaricae} \textbf{10} (1959),
  no.~3-4, 337--356.

\bibitem{ER59}
P.~Erd{\H o}s and A.~R{\'e}nyi, On random graphs {I}, \emph{Publ. Math.
  Debrecen} \textbf{6} (1959), 290--297.

\bibitem{ER60}
\bysame, On the evolution of random graphs, \emph{Publications of the
  Mathematical Institute of the Hungarian Academy of Sciences} \textbf{5}
  (1960), no.~1, 17--60.

\bibitem{ESW}
P.~Erd{\H o}s, S.~Suen, and P.~Winkler, On the size of a random maximal graph,
  \emph{Random Structures \& Algorithms} \textbf{6} (1995), no.~2-3, 309--318.

\bibitem{FFS}
J.~R. Faudree, R.~J. Faudree, and J.~R. Schmitt, A survey of minimum saturated
  graphs, \emph{Electron. J. Combin} \textbf{18} (2011), 36.

\bibitem{FGM}
G.~Fiz~Pontiveros, S.~Griffiths, and R.~Morris, The triangle-free process and
  ${R}(3, k)$, \emph{Memoirs of the American Mathematical Society}, to appear.

\bibitem{GSST}
S.~Gerke, D.~Schlatter, A.~Steger, and A.~Taraz, The random planar graph
  process, \emph{Random Structures \& Algorithms} \textbf{32} (2008), no.~2,
  236--261.

\bibitem{GKLO}
S.~Glock, D.~K\"uhn, A.~Lo, and D.~Osthus, On a conjecture of Erd\H os on locally sparse Steiner triple systems, \emph{arXiv
  preprint arXiv:1802.04227} (2018).

\bibitem{GRW}
C.~Greenhill, A.~Ruci{\'n}ski, and N.~C. Wormald, Random hypergraph processes
  with degree restrictions, \emph{Graphs and Combinatorics} \textbf{20} (2004),
  no.~3, 319--332.

\bibitem{JL}
S.~Janson and M.~J. Luczak, Susceptibility in subcritical random graphs,
  \emph{Journal of Mathematical Physics} \textbf{49} (2008), no.~12, 125207.

\bibitem{JLR}
S.~Janson, T.~{\L{}}uczak, and A.~Ruci{\'n}ski, \textbf{Random graphs},
  Cambridge University Press, 2000.

\bibitem{Kee}
P.~Keevash, Counting designs, \emph{Journal of the European Mathematical
  Society} \textbf{20} (2018),
  no.~4, 903--927.

\bibitem{KSV}
M.~Krivelevich, B.~Sudakov, and D.~Vilenchik, On the random satisfiable
  process, \emph{Combinatorics, Probability and Computing} \textbf{18} (2009),
  no.~5, 775--801.

\bibitem{KOT}
D.~K\"uhn, D.~Osthus, and A.~Taylor, On the random greedy ${F}$-free hypergraph
  process, \emph{SIAM Journal on Discrete Mathematics} \textbf{30} (2016),
  no.~3, 1343--1350.

\bibitem{Kwa}
M.~Kwan, Almost all {S}teiner triple systems have perfect matchings,
  \emph{arXiv preprint arXiv:1611.02246} (2016).

\bibitem{LR}
T.~{\L}uczak and A.~Ruci{\'n}ski, Tree-matchings in graph processes, \emph{SIAM
  Journal on Discrete Mathematics} \textbf{4} (1991), no.~1, 107--120.

\bibitem{OT}
D.~Osthus and A.~Taraz, Random maximal ${H}$-free graphs, \emph{Random
  Structures \& Algorithms} \textbf{18} (2001), no.~1, 61--82.

\bibitem{Picl}
M.~E. Picollelli, The final size of the ${C}_\ell$-free process, \emph{SIAM
  Journal on Discrete Mathematics} \textbf{28} (2014), no.~3, 1276--1305.

\bibitem{RW}
A.~Ruci{\'n}ski and N.~C. Wormald, Random graph processes with degree
  restrictions, \emph{Combinatorics, Probability and Computing} \textbf{1}
  (1992), no.~2, 169--180.

\bibitem{SW}
J.~Spencer and N.~Wormald, Birth control for giants, \emph{Combinatorica}
  \textbf{27} (2007), no.~5, 587--628.

\bibitem{War}
L.~Warnke, The ${C}_\ell$-free process, \emph{Random Structures \& Algorithms}
  \textbf{44} (2014), no.~4, 490--526.

\bibitem{Wol}
G.~Wolfovitz, Lower bounds for the size of random maximal ${H}$-free graphs,
  \emph{Electronic Journal of Combinatorics} \textbf{16} (2009), no.~1, \#R4.

\end{thebibliography}
\providecommand{\bysame}{\leavevmode\hbox to3em{\hrulefill}\thinspace}
\providecommand{\MR}{\relax\ifhmode\unskip\space\fi MR }
\providecommand{\MRhref}[2]{%
  \href{http://www.ams.org/mathscinet-getitem?mr=#1}{#2}
}
\providecommand{\href}[2]{#2}

\end{document}